\documentclass[10pt,reqno]{amsart}
\usepackage{amssymb,amsfonts,amsthm,amscd,stmaryrd,dsfont,esint,upgreek,constants,todonotes}
\usepackage[mathcal]{euscript}
\usepackage[latin1]{inputenc}   
\usepackage[mathcal]{euscript}
\usepackage{graphicx,color}
\numberwithin{equation}{section}
\newcommand{\N}{\mathbb N}

\newcommand{\R}{\mathbb R}
\def\E{\mathbb E}
\def\P{\mathbb P}
\newcommand{\Pw}{\mathcal P_2(\R^d)}
\newcommand{\Pk}{\mathcal P_1(\R^d)}
\newcommand{\linf}{L^{\infty}}
\newcommand{\sF}{\mathcal{F}}
\newcommand{\bP}{\mathbb{P}}
\newcommand{\sL}{\mathcal{L}}

\newcommand{\sG}{\mathcal{G}}
\newcommand{\sA}{\mathcal{A}}
\newcommand{\sP}{\mathcal{P}}
\newcommand{\bx}{\mathbf{x}}

\newconstantfamily{C}{symbol=C}
\newconstantfamily{c}{symbol=c}
\newconstantfamily{O}{symbol=\Omega}

\def\XXint#1#2#3{{\setbox0=\hbox{$#1{#2#3}{\int}$}
\vcenter{\hbox{$#2#3$}}\kern-.5\wd0}}

\numberwithin{equation}{section}
\newtheorem{thm}{Theorem}[section]
\newtheorem{lem}[thm]{Lemma}

\newtheorem{prop}[thm]{Proposition}

\theoremstyle{definition}

\newtheorem{rmk}[thm]{Remark}

\def\smallnegint{\mathop{\int\mkern-13mu
        \raise.5ex\hbox{${\scriptscriptstyle\diagup}$}}\nolimits}
\def\ds{\displaystyle}

\def\ep{\varepsilon}

\def\F{{\mathcal F}}

\def\bx{{\bf x}}
\def\by{{\bf y}}
\def\ssetminus{\,\raise.4ex\hbox{$\scriptstyle\setminus$}\,}

\newcommand{\be}{\begin{equation}}
\newcommand{\ee}{\end{equation}}

\newcommand{\bc}{\begin{case}}
\newcommand{\ec}{\end{cases}}
\newcommand{\bs}{\begin{split}}
\newcommand{\es}{\end{split}}

\newcommand{\norm}[1]{\left\Vert#1\right\Vert}
\newcommand{\abs}[1]{\left\vert#1\right\vert}

\newcommand{\vs}{\vskip.075in}

\renewcommand{\bar}{\overline}
\renewcommand{\tilde}{\widetilde}
\renewcommand{\hat}{\widehat}

\begin{document}
\title[Convergence rate for the optimal control of McKean-Vlasov dynamics]{An algebraic convergence rate for the optimal control of McKean-Vlasov dynamics}

\author[Pierre Cardaliaguet, Samuel Daudin, Joe Jackson and Panagiotis E. Souganidis]
{Pierre Cardaliaguet, Samuel Daudin, Joe Jackson and Panagiotis E. Souganidis}
\address{Universit\'e Paris-Dauphine, PSL Research University, Ceremade, 
Place du Mar\'echal de Lattre de Tassigny, 75775 Paris cedex 16 - France}
\email{cardaliaguet@ceremade.dauphine.fr }
\address{Universit\'e Paris-Dauphine, PSL Research University, Ceremade, 
Place du Mar\'echal de Lattre de Tassigny, 75775 Paris cedex 16 - France}
\email{daudin@ceremade.dauphine.fr }
\address{Department of Mathematics, The University of Texas at Austin, Austin, Texas 78712, USA}
\email{jjackso1@utexas.edu}
\address{Department of Mathematics, University of Chicago, Chicago, Illinois 60637, USA}
\email{souganidis@math.uchicago.edu}
\vskip-0.5in 

\dedicatory{Version: \today}

\begin{abstract}
We establish an algebraic rate of convergence in the large number of particles limit of the value functions of $N$-particle stochastic control problems towards the value function of the corresponding McKean-Vlasov problem, also known as mean field control.  The rate is obtained  in the presence of both idiosyncratic and common noises and in a setting where the value function for the McKean-Vlasov problem need not be smooth. Our approach relies crucially on uniform in $N$ Lipschitz and semi-concavity estimates for the $N$-particle value functions as well as a certain concentration inequality. 
\end{abstract}

\maketitle

\section{Introduction}

We consider an optimal control problem with a large number of particles. The value function for this optimization problem reads
\be \label{def.VN}
\mathcal V^N(t_0, \bx_0):= \inf_{\alpha \in \mathcal A^N} 
 \E\Big[ \int_{t_0}^T (\frac{1}{N}\sum_{k=1}^N L(X^k_t,\alpha^k_t) +\mathcal F(m^N_{{\bf X}_t}))dt +
\mathcal G(m^N_{{\bf X}_T})\Big ],
\ee
where $T>0$ is a finite horizon, $t_0\in [0,T]$ is the initial time, and $\bx_0=(x_0^1, \dots, x_0^N)\in (\R^d)^N$ is the initial position of the $N$ particles. The infimum is taken over the set $\mathcal A^N$ of progressively measurable $(\R^d)^N$-valued processes $\alpha = (\alpha^k)_{k = 1}^N$ in $L^2([0,T] \times \bar \Omega ; (\R^d)^N)$ and ${\bf X}= (X^1, \dots, X^N)$ satisfies, for each $k\in \{1, \dots, N\}$,

\be\label{takis1}
X^k_t = x^k_0+\int_{t_0}^t \alpha^k_s ds +\sqrt{2}(B^k_t-B_{t_0}^k) + \sqrt{2a_0} (B^0_t-B_{t_0}^0) \qquad t\in [t_0,T].
\ee

The $(B^k)_{k \geq 0}$ are independent $d$-dimensional Brownian motions defined on the fixed filtered probability space $(\bar{\Omega}, \bar{\sF}, \bar{\mathbb{F}}, \bar{\bP})$ satisfying the usual conditions, and $L^2([0,T] \times \bar \Omega ; (\R^d)^N)$ denotes the set of square-integrable and progressively measurable processes taking values in $(\R^d)^N$,
and $m^N_{{\bf X}_t}$ is the empirical measure of ${\bf X}_t$.
The cost function $L:\R^d\times \R^d\to \R$ is supposed to be convex in the second variable and smooth while the maps $\mathcal F, \mathcal G:\Pk\to \R$ are assumed to be smooth and bounded over the space $\Pk$ of Borel measures on $\R^d$ with a finite first-moment (precise assumptions will be given in section \ref{sec.mainresult}).
The constant $a_0\geq 0$ is the level of the common noise, and the $(B^k)_{k\geq 1}$ are viewed as independent or idiosyncratic noises.

\subsection{Our results} To describe our result we need to introduce the map $\mathcal U :[0,T] \times \sP_2(\R^d) \to \R$, where  $\sP_2(\R^d)$ is the space of Borel measures on $\R^d$ with a finite second-moment, given, for $(t_0, m_0)\in [0,T]\times \sP_2(\R^d)$, by
\be\label{def.U}
\mathcal U(t_0,m_0):=
\inf_{\alpha \in \mathcal A} \E[\int_{t_0}^T\big( L(X_t, \alpha_t(X_t)) + \sF( \sL(X_t | \sF^{B^0}_t))\big) + \sG(\sL(X_T | \sF^{B^0}_T))],
\ee
where  the infimum is taken over an appropriate set $\mathcal A$ of admissible controls (this will be made precise later), $\mathbb{F}^{B^0} = (\sF_{t}^{B^0})_{0 \leq t \leq T}$ denotes the filtration generated by $B^0$, $ \sL(X_t | \sF^{B^0}_t)$ is the law of $X_t $ conditioned upon  $\sF_{t}^{B^0}$,  and
\be
X_t= \bar X_{t_0}+\int_{t_0}^t \alpha_s(X_s) ds +\sqrt{2}(B_t - B_{t_0}) + \sqrt{2a_0} (B^0_t - B^0_{t_0}),
\ee
with $B$ another Brownian motion, $\bar X_{t_0}$ a random initial condition with law $m_0$ and  $B^0$, $B$ and  $\bar X_{t_0}$ mutually independent.
\vs

Although it is known that, as   $N$ tends to infinity,  $\mathcal V^N$ converges to $\mathcal U$, the existing convergence results come without any rate

\vs
Our main result is the following algebraic convergence rate: there exists $\beta\in (0,1]$, depending only on the  dimension $d$,  and $C>0$, depending on the data $(\mathcal{F},\mathcal{G},H)$,  such that, for any $(t,\bx)\in [0,T]\times (\R^d)^N$, 
\be
\abs{\mathcal V^N(t, \bx) - \mathcal U(t, m_{\bx}^N)} \leq CN^{-\beta} (1 + M_2(m^N_{\bx})),
\ee
where $M_2(m^N_{\bx})= N^{-1}\sum_{i=1}^N|x^i|^2$ is the second-order moment of the measure $m^N_{\bx}$.
\vs
Although the exact value of $\beta$ could be traced back through the computation, it is clearly not optimal. In particular, it is very far from the one obtained for  a standard particle system. Similarly,  even if some dependence with respect to a moment of the measure $m^N_{\bx}$ is expected, the dependence given here is probably far from sharp.

\subsection{Background and related literature} The convergence of $\mathcal V^N$  to $\mathcal U$ was shown by Lacker  \cite{La17}  in a general framework and  for suitable initial data but without common noise, that is, with  $a_0=0$ in \eqref{takis1}.
Recently, the results of \cite{La17} were  extended in Djete, Possama\"i and Tan \cite{DPT20} to problems with a common noise and interaction through the controls.
Beside \cite{DPT20, La17} several other papers have studied the question of the mean field limit of optimal control problems, for example,  Cavagnari, Lisini, Orrieri and Savar\'e \cite{CaLiOrSa} and Fornasier, Lisini, Orrieri and Savar\'e  \cite{FLOS19} investigate the problem without noise using $\Gamma-$convergence techniques. The recent contribution of Gangbo, Mayorga and Swiech \cite{GMS} studies the mean field limit without idiosyncratic but  with common noise using partial differential equations (PDE for short) techniques. This is possible thanks to the fact that $\mathcal V^N$ solves the Hamilton-Jacobi (HJ for short) equation
\be\label{eq.HJBN}
\left\{\begin{array}{l}
\ds -\partial_t \mathcal V^N(t,{\bf x}) -\sum_{j=1}^N \Delta_{x^j}\mathcal V^N(t,{\bf x}) -a_0 \sum_{i,j=1}^N {\rm tr} (D^2_{ij}\mathcal V^N(t,\bx))\\
\ds \qquad  +\frac1N\sum_{j=1}^N H(x^j, N D_{x^j}\mathcal V^N(t, {\bf x}))=\mathcal F(m^N_{{\bf x}})\ \  {\rm in}\ \  (0,T)\times (\R^d)^N, \\
\ds  \mathcal V^N(T,{\bf x})= \mathcal G(m^N_{{\bf x}})\ \  {\rm in}\ \ (\R^d)^N,
\end{array}\right.
\ee
where
$
H(x,p)= \sup_{\alpha\in \R^d}[ - p\cdot \alpha-L(x,\alpha)],
$
while $\mathcal U$ is expected to solve (in some sense) the infinite dimensional HJ equation
\be\label{eq.HJcnintro}
\left\{\begin{array}{l}
\ds -\partial_t \mathcal U(t,m)-(1+a_0)\int_{\R^d}{\rm div}_y( D_m\mathcal U(t,m,y))m(dy) \\
\ds \hspace{2cm}  - a_0\int_{\R^{2d}} {\rm tr}(D^2_{mm}\mathcal U(t,x,m,y,y'))m(dy)m(dy')\\
\ds \hspace{1cm} +\int_{\R^d} H(y, D_m\mathcal U(t,m,y))m(dy)=\mathcal F(m)\ \  {\rm in}\ \ (0,T)\times \Pw,\\
\ds \mathcal U(T,m)= \mathcal G(m)\ \ {\rm in}\ \ \Pw.
\end{array}\right.
\ee
For the definition of the derivatives $D_m\mathcal U$ and $D^2_{mm}\mathcal U$ we refer to the books of Cardaliaguet, Delarue, Lasry and Lions \cite{CDLL} and Carmona and Delarue \cite{CaDeBook}.

One of the reasons for introducing the value functions is that they provide  optimal feedbacks for the optimization problems.  For the particle system, this optimal feedback is given (rigorously) by
$
\alpha^*_i(t,\bx)= -D_pH(x_i, ND_{x_i} \mathcal V^N(t, \bx)),
$
while for the limit system it takes the form (at least formally)
$
\alpha^*_t(x,m)= -D_pH(x, D_m\mathcal U(t,m,x)).
$
The difficulty in the PDE analysis of \cite{GMS} is that, in the absence of the idiosyncratic noise, the value function $\mathcal V^N$ is  not smooth in general, and, thus,  \eqref{eq.HJBN} has to be interpreted in the  viscosity sense.  A suitable notion of viscosity solution for the infinite dimensional HJ equation \eqref{eq.HJcnintro} without idiosyncratic noise is introduced in \cite{GMS} , and then is proven that $\mathcal V^N$ converges to this viscosity solution. In the presence of idiosyncratic noise the notion of viscosity solution to \eqref{eq.HJcnintro} is not understood yet and we will not try to use this approach. 

This being said, we believe that our technique of proof applies when the idiosyncratic noise is degenerate. Indeed, none of the estimate on $\mathcal{V}^N$ that we need relies on its $\mathcal{C}^2$ regularity. Although the non-degeneracy of the idiosyncratic noise is regularizing at the level of the N-particle system, it does not facilitate the analysis of the limit problem for which the value function is not necessarily smooth. We emphasise that the presence of the idiosyncratic noise is the main difficulty that we want to address in this paper. In particular it prevents any use of viscosity solutions techniques in Hilbert spaces as in \cite{GMS}.


\subsection{More about our results} While the existing results mentioned above demonstrate the convergence of $\mathcal{V}^N$ to $\mathcal{U}$ under many different technical hypotheses and using a variety of techniques, none provides a rate of convergence. Our main result fills this gap in the literature, by providing a rate of convergence of $\mathcal V^N$ to $\mathcal U$ in the presence of both idiosyncratic and common noise.

We emphasise that quantitative information about the convergence toward the mean-field limit is particularly important for numerical applications. Obtaining a convergence of the value functions with a rate also happens to be a useful starting point in order to prove finer propagation of chaos results, as illustrated in \cite{CaSou22} which is based on the results of the present paper.


The primary challenge we face is related to the (lack of) regularity of $\mathcal{U}$.  Indeed, if $\mathcal{U}$ is a smooth solution solution to \eqref{eq.HJcnintro}, then the projections $\mathcal{U}^N : [0,T] \times (\R^d)^N \to \R$ given by $\mathcal{U}^N(t,\bx) = \mathcal{U}(t,m_{\bx}^N)$ are smooth solutions of the HJ equation
\be\label{eq.HJBNapprox}
\left\{\begin{array}{l}
\ds -\partial_t \mathcal U^N(t,{\bf x}) -\sum_{j=1}^N \Delta_{x^j}\mathcal U^N(t,{\bf x}) -a_0 \sum_{i,j=1}^N {\rm tr} (D^2_{ij}\mathcal V^N(t,\bx))\\
\ds  +\frac1N\sum_{j=1}^N H(x^j, N D_{x^j}\mathcal U^N(t, {\bf x}))
 =\mathcal F(m^N_{{\bf x}}) + E_N(t,\bx) \ \ {\rm in}\ \  (0,T)\times (\R^d)^N, \\
\ds  \mathcal U^N(T,{\bf x})= \mathcal G(m^N_{{\bf x}})\ \  {\rm in} \ \ (\R^d)^N,
\end{array}\right.
\ee
with $E_N(t,\bx) = -N^{-2} \sum_{ j = 1}^N {\rm tr} (D_{mm} \mathcal{U}(t,m_{\bx}^N,x_i,x_i)).$

If $D_{mm} \mathcal{U}$ is bounded, then it is immediate that $|E_n| = O(1/N)$. Thus,  $\mathcal{U}^N$ solves the same equation as $\mathcal{V}^N$ up to a term of order $O(1/N)$. By a comparison argument, we conclude that $|\mathcal{U} - \mathcal{V}| = O(1/N)$, that is,  there exists a constant $C$ such that, for all  $t \in [0,T]$ and $\bx \in (\R^d)^N$,
  $  |\mathcal{V}^N(t,\bx) - \mathcal{U}(t,m_{\bx}^N)| \leq C/N. $
See also \cite{gpw21} for more on what convergence results can be obtained once \eqref{eq.HJcnintro} has a sufficiently smooth solution. This argument is similar to the approach taken in \cite{CDLL, CaDeBook} to study
the convergence problem in the context of mean field games (see Lasry and Lions  \cite{LLJapan}) in situations where a classical solution to the so-called master equation is known to exist; also see  Bayraktar and Cohen \cite{BaCo18} and Cecchin and Pelino \cite{CePe19} for related results. In this setting, convergence is related to the propagation of chaos for the optimal trajectories of the game.

Of course, the simple argument outlined above works only when the value function $\mathcal{U}$ is smooth. For instance,  this would is the case if the maps $\mathcal F$ and $\mathcal G$ are convex and sufficiently smooth (see the discussion in Chap. 3.7 of \cite{CDLL}). However,  we do not assume such a convexity property and the map $\mathcal U$ is expected to present discontinuities in its first-order derivative, as  can be seen in, for instance, Briani and Cardaliaguet \cite{BrCa}. Because of this, the techniques in \cite{CDLL, CaDeBook} break down.

When the value function is not smooth, the convergence rate has been studied primarily in the case of finite state space; see  Kolokoltsov \cite{Ko12} and Cecchin \cite{Ce21}. In this finite state space setting, the convergence rate is  of order $1/\sqrt{N}$. Indeed, as explained in \cite{Ce21},  the particle system is then a kind of discretization of the continuous McKean-Vlasov equation.


The situation is different and much more difficult in the continuous state space setting. This might come as a surprise since the convergence rate for particle systems is very well understood; see, for instance,  Fournier and Guillin,  \cite{FoGu}. The main difficulty, however,   is that, even though the optimal feedback in the particle system remains bounded independently of $N$ (see Lemma \ref{lem.estiVN}), it cannot be expected to be uniformly continuous as a function of the empirical measure. Indeed, this uniform continuity would imply the $C^1-$regularity of the limit $\mathcal U$, which does not hold in general. So it is necessary to find a way to show that, despite the fact that the controls played by each particle might be very different, a kind of concentration of measure takes place.

Finally, we mention that a result similar to the one we prove here in the context of Mean-Field Control, remains an open question in the context of Mean-Field Games. The difficulty is that the $N$-player game is described this time by a system of $N$ coupled HJB equations, instead of just one HJB equation in the present case. And it proves difficult to obtain estimates on the PDE system which are uniform in the number of players.

%
\subsection{Strategy of the proof}
We discuss briefly the strategy of the proof.  We first point out that we do  not rely on a propagation of chaos, which we cannot prove at this stage. Indeed, as for a given initial condition there might be several optimal trajectories for the limit problem, a propagation of chaos is not expected to hold without additional assumptions on the initial data. The main ingredients for the proof are,  uniform in $N$,  Lipschitz and semiconcavity estimates for $\mathcal V^N$, and a concentration inequality. To bound from above $\mathcal V^N$ by $\mathcal U$ is relatively easy, because $\mathcal V^N$ can be transformed into an approximate subsolution for the Hamilton-Jacobi equation \eqref{eq.HJcnintro}. The opposite inequality is much trickier, because it  seems impossible to transform an optimal control for the $\mathcal V^N$,  in which the control depends on each particle,  into a feedback for $\mathcal U$. We overcome this difficulty by dividing the particles into subgroups in such a way that the optimal controls for the particles in each subgroup are close and show a propagation of chaos, based on a concentration inequality, for each subgroup. The proof being technical, we first show the result when there is no common noise, and, in a second step, extend the result to problems with  common noise.

\subsection{Organization of the paper}
In the rest of the introduction we fix notation. We state the assumptions and the main result in section \ref{sec.mainresult}. As the proof of the convergence rate is technical, we start in section \ref{sec.withoutCN} with the problem without common noise. Indeed this case contains the main ideas without the extra technicalities due to the common noise. We first give some estimates on $\mathcal V^N$ and $\mathcal U$ (subsection \ref{subsec.regu}), then show the relatively easy bound from above for $\mathcal V^N$ in subsection \ref{subsec.easy}. The main part of the proof, that  is, the bound from below,  which is the aim of subsection \ref{subsec.main} requires a concentration inequality proved in subsection \ref{subsec.ProofLemma}. We explain the adaptation  of the proof to the case with common noise in section \ref{sec.CN}.


\subsection{Notation} \label{subsec:assumptions}
We work on $\R^d$, write $I_d$ for the identity matrix in $\R^d$, and $B_R$ for the ball in $\R^d$ centered at the origin with radius $R$.  For $\bx = (x^1,...,x^N) \in (\R^d)^N$, $m_{\bx}^N \in \sP(\R^d)$ stands for  the empirical measure of $\bx$, that is, $m_{\bx}^N = \frac{1}{N} \sum_{i = 1}^N \delta_{x^i}$.
If $\varphi: [0,T] \times \R^d \rightarrow \R^d$ is smooth enough, we write $D\varphi$, $\Delta \varphi$ and $D^2 \varphi$ for the derivatives with respect to space and  $\partial_t \varphi$ and $\partial_{tt} \varphi$ the derivatives with respect to time. Similarly, for $\mathcal{V} = \mathcal{V}(t,x^1,...,x^N) : [0,T] \times (\R^d)^N \rightarrow \R$, we define the derivatives $D_{x^k}\mathcal{V}$, $\Delta_{x^k} \mathcal{V}$, $\partial_t \mathcal{V}$.   We denote by $\sP(\R^d)$ the set of Borel probability measures on $\R^d$ and note that, if $m\in \sP(\R^d)$ has a density, for  simplicity of notation, $m$ is  also used to denote the density.  Given $m \in \sP(\R^d)$ and $p \geq 1$, $M_p(m)$ is the $p^{th}-$moment of $m$, that is,  $M_p(m) = \int_{\R^d} |x|^p dm$, and   $\sP_p(\R^d)$ the set of $m \in \sP(\R^d)$ such that $M_p(\R^d) < \infty$. We endow $\sP_p(\R^d)$ with the Wasserstein metric ${\bf d}_p$, defined by
   $ {\bf d}^p_p(m, m') := \inf_{\pi \in \Pi(m,m')} \int_{\R^d} |x - y|^p d\pi(x,y),$
where $\Pi(m,m')$ is the set of all $\pi \in \sP(\R^d \times \R^d)$ with marginals $m$ and $m'$. Let  ${\bf L}$ be  the set of all 1-Lipschitz functions from $\R^d$ to $\R$. We recall the duality formula
   $ {\bf d}_1(m,m') = \sup_{\phi \in {\bf L}} \int_{\R^d} \phi d(m- m').$
%
For $\mathcal{U} : \mathcal{P}_1(\R^d) \rightarrow \R$ smooth enough, $\dfrac{\delta U}{\delta m} : \mathcal{P}_1(\R^d) \times \R \rightarrow \R$ denotes the  linear functional derivative, which satisfies, for all $m,m' \in \mathcal{P}_1(\R^d)$ and all $h \in (0,1)$,
\newline
$ \mathcal{U}(m') - \mathcal{U}(m) = \int_0^1 \int_{\R^d} \frac{\delta U}{\delta m}((1-h)m+hm',x)(m'-m)(dx)dh.$
We  use  the standard convention $\int_{\R^d} \dfrac{\delta U}{\delta m}(m,x)m(dx)=0$ for all $m \in \mathcal{P}_1(\R^d)$. If $\dfrac{\delta \mathcal{U}}{\delta m}$ is differentiable with respect to the space variable, we define the $L$-derivative of $\mathcal U$ by $D_mU(m,x) =D_x\dfrac{\delta U}{\delta m}(m,x)$. Higher order derivatives are defined similarly.

We refer to \cite{CDLL} Chapter 2 and  \cite{CaDeBook} Book 1, Chapter 5 for the properties of the $L$-derivatives.

Finally, throughout the paper we use $C$ for positive constants that depend, unless otherwise noted,  on the data and may change from line to line with this being made explicit.
%
%
%
\section{Assumptions and main result}\label{sec.mainresult}
\subsection{Assumptions}
We now state our standing assumptions on the maps $H, F$ and $G$, which constitute the data of our problem. We keep in mind that $L:\R^d\times \R^d\to \R$ is a Legendre transform of $H$ with respect to the last variable, that is, $L(x,a)= \sup_{p\in \R^d}[ -a\cdot p-H(x,p)].$
We assume that
\be \label{hyp.quad}
\begin{cases}
 \text{$H\in C^2(\R^d\times \R^d;\R)$ and for some  $c$, $C > 0$ and  all $(x,p) \in \R^d \times \R^d,$}  \\
 - C + c |p|^2 \leq H(x,p) \leq C + \frac{1}{c} |p|^2 \ \ \text{and} \ \ |D_xH(x,p)| \leq C(|p|+1).
\end{cases}
\ee
\be\label{convex}
\begin{cases}
\text {$H$ is locally strictly convex with respect to the last variable,}\\
\text{that is, for any $R>0$, there exists $c_R>0$ such that}\\
D^2_{pp}H(x,p)\geq c_R I_d\ \  \text{for all} \ \  (x,p)\in \R^d\times  \overline{B_R}, 
\end{cases}
\ee
\be\label{hyp.growth2}
\begin{cases}
\text{for any $R>0$, there exists $C_R>0$ such that}\\
|D^2_{xx}H(x,p)| + |D^2_{xp}H(x,p)|\leq C_R \ \  \text{for all} \quad (x,p)\in \R^d \times   \overline{B_R},
\end{cases}
\ee
\be\label{F}
\text{$\mathcal F \in C^2(\Pk; \R)$ with $\mathcal F$, $D_m\mathcal F$, $D^2_{ym}\mathcal F$ and $D^2_{mm}\mathcal F$ uniformly bounded,}
\ee
and, finally,
\be\label{G}
\text{$\mathcal G\in C^4(\Pk; \R)$ with all derivatives up to order $4$ uniformly bounded.}
\ee
For simplicity, in what follows we put together all the assumptions above in
\be\label{ass.main}
\text{assume that \eqref{hyp.quad}, 
\eqref{convex},
\eqref{hyp.growth2},  \eqref{F} and \eqref{G} hold,}
\ee

%
%

\begin{rmk}
We make the following comments regarding  \eqref{ass.main}.

(i)~The strict convexity of $H$ with respect to the  gradient  variable is standard in optimal control. In particular, it implies that $L$ has the same regularity as $H$.

(ii) Although  the at most  linear growth in $p$ of $D_xH$, which is used to obtain, independent of $N$,  Lipschitz estimates on the value function $\mathcal V^N$ (see Lemma \ref{lem.estiVN}),  is somehow restrictive,  we do not know if it is possible to avoid it. It is, however,  satisfied by, for instance,  a Hamiltonian of the form $H(x,p)= |p|^2+ V(x)\cdot p$ for some smooth and globally Lipschitz continuous vector field $V:\R^d\to \R^d$. 


(iii)~The fact that the ``full'' Hamiltonian $(x,p, m)\to H(x,p)-\mathcal F(m)$ has a separate form is not completely necessary. In particular, our method allows to handle dynamics of the form 
$$ dX_t = b(X_t,\mathcal{L}(X_t|\mathcal{F}_t^{B^0}))dt + \alpha_tdt + \sqrt{2}dB_t + \sqrt{2a_0}dB_t^0$$ 
for some bounded nonlinear drift $b: \R^d \times \mathcal{P}_2(\R^d) \rightarrow \R^d$ with bounded derivatives. However this leads to much heavier computations that we decided to avoid to keep the paper as clear as possible.

(iv)~The uniform bounds on $D_m\mathcal F$ and $D_m\mathcal G$ imply that both maps are Lipschitz continuous in $\Pk$. The additional smoothness  is used to obtain, independent of $N$, semiconcavity estimates on the value function $\mathcal V^N$ (see Lemma \ref{lem.semiconcesti}).

(v)~As $L$ is the Legendre transform of $H$, \eqref{convex} implies, after a simple calculation, that,  for any $R > 0$, there exists $C_R > 0$ such that
\be \label{hyp.l}
|D_a L(x,a)| \leq C_R \ \  \text{for all} \ \ (x,a) \in \R^d \times \overline{B_R}.
\ee
\end{rmk}

\subsection{The formulation of the problem}
For concreteness, we fix throughout the paper a filtered probability space $(\bar{\Omega}, \bar{\sF}, \bar{\mathbb{F}}= (\bar{\sF})_{t \geq 0}, \bar{\bP})$ satisfying the usual conditions and hosting independent $d$-dimensional Brownian motions $B^0$ and $(B^k)_{k \in \N}$.

\subsubsection{The definition of $\mathcal{V}^N$}
The definition of $\mathcal{V}^N$ and the relevant quantities/functions were  given and discussed  in the introduction--see  \eqref{def.VN} and  \eqref{takis1},
where it was also explained that, assuming  \eqref{ass.main},
$\mathcal{V}^N$ is the unique classical solution to the Hamilton-Jacobi equation \eqref{eq.HJBN} and that the infimum in \eqref{def.VN} is achieved (in feedback form) by the function $\alpha = (\alpha^k)_{k = 1}^N : [0,T] \times (\R^d)^N \to \R^N$ given by
\be\label{takis11}
\alpha_k(t,\bx) = - D_p H(x^k, ND_{x^k} \mathcal{V}^N(t,\bx)).
\ee

\subsubsection{The definition of $\mathcal{U}$ without common noise}
Suppose now that $a_0 = 0$. To define $\mathcal{U}$, it is  more intuitive to work with closed-loop controls, and to view the problem in terms of deterministic control of the associated Fokker-Planck equation.

For fixed $(t_0, m_0) \in [t_0,T] \times \sP_2(\R^d)$, let  $\mathcal{A}(t_0, m_0)$ be the set of pairs $(m,\alpha)$ with $m = (m_t)_{t \in [t_0,T]} = (m(t, \cdot))_{t \in [t_0,T]} \in C^0([t_0,T];\sP_2(\R^d))$, $\alpha : [t_0,T] \times \R^d \to \R^d$ measurable such that $\int_{t_0}^T \int_{\R^d} |\alpha(t,x)|^2 m(t,dx)dt < \infty$ and $m$ solves (in the sense of distributions) the Fokker-Planck equation
   $$ \partial_t m = \Delta m - \text{div}(m \alpha) \ \ \text{in} \ \ (t_0, T]\times \R^d \ \ \ \text{and}  \ \ m(t_0,\cdot) = m_0.$$

Then we define $\mathcal{U} : [0,T] \times \sP_2(\R^d) \to \R$ by
\be\label{def.uwocommon}
\begin{split}
    &\mathcal U(t_0,m_0) =\\
   & \inf_{(m,\alpha) \in \sA(t_0,m_0)} \Big\{ \int_{t_0}^T \big( \int_{\R^d} L(x, \alpha(t,x)) m(t,dx) + \sF(m_t)) dt + \sG(m_T) \Big\}.
 \end{split}
 \ee

Notice that it is not restrictive to consider feedback controls which are only function of the time and space variables instead of controls which depend as well on the probability measure $m(t)$. Indeed, if $\tilde{\alpha}:[0,T] \times \R^d \times \mathcal{P}_2(\R^d)$ is such control with corresponding trajectory $\tilde{m} \in \mathcal{C}([0,T],\mathcal{P}_2(\R^d))$, we can obtain the curve $t \mapsto \tilde{m}(t)$ with the same cost by considering the control $\alpha(t,x) = \tilde{\alpha}(t,x,\tilde{m}(t))$.

One advantage to using this deterministic formulation of the McKean-Vlasov control problem is that, at least in the absence of common noise, the dynamic programming principle is straightforward. In particular, we can assert the following, which will be useful in what follows.
\begin{prop} \label{pro.dppnocommon}
Assume \eqref{ass.main}.  Then, for any $0 \leq t_0 \leq t_1 \leq T$,  
\begin{align*}
    \mathcal{U}(t_0,m_0) = \inf_{(m,\alpha) \in \sA(t_0,m_0)}\Big\{ \int_{t_0}^{t_1} \big( \int_{\R^d} L(x, \alpha(t,x)) m_t(dx) + \sF(m_t)) dt + \mathcal{U}(t_1, m_{t_1}) \Big\}.
\end{align*}
\end{prop}

\subsubsection{The definition of $\mathcal{U}$ with common noise}
To define $\mathcal{U}$ when $a_0 > 0$, we use again a form of closed-loop formulation, but this time the relevant Fokker-Planck equation becomes stochastic and we work with a notion of weak solution.
\vs
For fixed $(t_0,m_0) \in [0,T] \times \sP_2(\R^d)$, we define a control rule $\mathcal{R} \in \mathcal{A}(t_0,m_0)$ to be a tuple
    $\mathcal{R} = (\Omega, \sF, \mathbb{F}, \bP, W, m ,\alpha), $
where $(\Omega, \sF,  \mathbb{F}) = (\sF_t)_{0 \leq t \leq T},\bP)$ is a filtered probability space supporting the $d$-dimensional Brownian motion $W$,  $\alpha = (\alpha_t)_{t_0 \leq t \leq T}$ is a $\mathbb{F}$-progressively measurable process taking values in $\linf(\R^d ; \R^d)$ and such that $\alpha$ is uniformly bounded, in the sense that \begin{align}
        \| { \sup_{t \in [t_0,T]} {\|{\alpha_t}\|}_{L^{\infty}(\R^d; \R^d)}}\|_{L^{\infty}(\Omega)} < \infty.
    \end{align}
and  $m$ satisfies the stochastic partial differential equation
\be\label{eq.mtcn}
\begin{split}
dm_t(x) &= \left[(1+a_0)\Delta m_t(x)-{\rm div}(m_t \alpha_t(x))\right]dt\\
&+ \sqrt{2a_0}Dm_t(x) \cdot dW_t \ \ \text{in} \ \ (t_0, T]\times \R^d \ \ \text{with} \ \ m_{t_0}=m_0 \ \ \text{in} \ \ \R^d.
\end{split}
\ee
The last condition means that, $\P-$a.s., for any smooth test function $\phi\in C^\infty([0,T]\times \R^d)$ with a compact support and for any $t\in [t_0,T]$,
\begin{align*}
& \int_{\R^d} \phi_t(x)m_t(dx)  = \int_{\R^d} \phi_0(x)\bar m_0(dx)+ \int_{t_0}^t \int_{\R^d} (\partial_t \phi_s(x)
 + \alpha_s(x)\cdot D\phi_s(x)\\
 &+ (1+a_0)\Delta \phi_s(x))m_s(dx)ds
  + \int_{t_0}^t \sqrt{2a^0}\int_{\R^d} D\phi_s(x)m_s(dx)\cdot dW_s.
 \end{align*}
Now we define
\be \label{def.ucommon}
\begin{split}
  &  \mathcal U(t_0,m_0) = \\
  &  \inf_{\mathcal{R} \in \sA(t_0,m_0)} \E^{\bP} \Big[  \int_{t_0}^T \big( \int_{\R^d} L(x, \alpha_t(x)) m_t(dx) + \sF(m_t)) dt + \sG(m_T) \Big].
\end{split}
\ee

The connection to the informal description \eqref{def.U} of $\mathcal{U}$ is that, if $\alpha$ is a bounded $L^\infty(\R^d ; \R^d)$-valued process defined on some filtered probability space probability space $(\Omega, \sF, \mathbb{F} = (\sF_t)_{0 \leq t \leq T}, \bP)$ supporting independent Brownian motions $B$ and $W$, $\alpha$ is a adapted to the filtration of $W$ and $X$ is a strong solution to the McKean-Vlasov equation
\begin{align} \label{xdynamics}
X_t = \bar X_{t_0}+\int_{t_0}^t \alpha_s(X_s) ds +\sqrt{2}(B_t - B_{t_0}) + \sqrt{2a_0} (W_t - W_{t_0}),
\end{align}
then $(\Omega, \sF, \mathbb{F}^W ,W,m,\alpha) \in \mathcal{A}(t_0,m_0)$, where $m_t = \sL(X_t | W)$, that is,  $m$ is the conditional law of $X$ given the filtration of the Brownian motion $W$.

As in the case $a_0 = 0$, we have the following  dynamic programming principle.
\begin{prop} \label{prop.dpcommon}
Assume \eqref{ass.main}. Then, for any $0 \leq t_0<t_1\leq T$,  for  $\mathcal{U}$ defined by \eqref{def.ucommon}, we have
$$
\mathcal U(t_0,m_0)= \inf_{( m,  \alpha) \in \mathcal{A}(t_0,m_0)} \E^{\bP}\left[ \int_{t_0}^{t_1} (\int_{\R^d} L(x, \alpha_t(x))m_t(dx)+\mathcal F(m_t))dt +\mathcal U(t_1,m_{t_1})\right].
$$
\end{prop}
Unlike in the case without common noise, where the control problem is deterministic and thus the dynamic programming principle is straightforward, in the common noise case we will need to use some machinery from Djete, Possama\"i and Tan \cite{DPT19} and  Lacker, Sholnikov and Zhang \cite{LSZ20} to verify that the dynamic programming principle holds.
To streamline the presentation, we present the proof of Proposition \ref{prop.dpcommon} as well as of some other technical results from \cite{DPT19, DPT20, LSZ20} in the Appendix.

\begin{rmk}
We could have defined $\mathcal{U}$ using \eqref{def.ucommon} when $a_0 = 0$ as well, and, in the end, it would be possible, thanks in part to Lemma \ref{It'^o} below, to prove that this is equivalent to \eqref{def.uwocommon}. We chose to define things separately with and without common noise mostly to avoid some unnecessary technicalities and to simplify the presentation for the reader interested in the case without common noise. The only mathematical reason for splitting up the definitions is that, for technical reasons,  it is convenient to work with $L^{\infty}-$feedback controls in the case of common noise, whereas without common noise we have no difficulty working with square-integrable controls.
\end{rmk}

\subsection{The main result}
With $\mathcal{V}^N$ defined by $\eqref{def.VN}$, $\mathcal{U}$ defined by \eqref{def.uwocommon}, if $a_0 = 0$, or \eqref{def.ucommon}, if $a_0 > 0$, we have the following result.

\begin{thm} \label{thm.main1}  Assume \eqref{ass.main}.  Then  there exists $\beta\in(0,1]$ depending  only on $d$ and $C>0$ depending on the data ($\mathcal{F}, \mathcal{G}, H$) such that, for any $(t,\bx)\in [0,T]\times (\R^d)^N$,
\newline
$
\left| \mathcal V^N(t, \bx)- \mathcal U(t, m^N_{\bx})\right| \leq C {N^{-\beta}} (1+ M_2(m^N_{\bx})).
$
\end{thm}
\vs
%
%

For the convenience of the reader we repeat here the strategy of the proof. We detail in section \ref{sec.withoutCN} the proof of Theorem \ref{thm.main1} when $a_0=0$,  the adaptation to the case $a_0>0$ being the aim of section \ref{sec.CN}. The proof of Theorem \ref{thm.main1} requires several steps: We first obtain  uniform in $N$ regularity  (Lipschitz and semiconcavity) estimates on $\mathcal V^N$ 
in Lemma~\ref{lem.estiVN} and Lemma~\ref{lem.semiconcesti} respectively.  Then we show how to bound from above $\mathcal V^N$ by $\mathcal U$ plus an error term (Proposition \ref{lem.ineqeasy}). This estimate is relatively easy and boils down to transforming the map $\mathcal V^N$ into a subsolution of the HJ equation \eqref{eq.HJcnintro}.  The converse estimate, which  is more involved,  is the aim of Proposition \ref{prop.hardineq}. The technical reason is that we found no way to embed $\mathcal U$ into the equation for $\mathcal V^N$ as a subsolution. Actually, since  $\mathcal U$ is semiconcave, it is naturally a supersolution of that equation and the remaining term is a priori large. We overcome this issue by using locally optimal feedback of the $N-$problem for the continuous one, the main difficulty being to compare the empirical measure in the $N-$problem to the solution of the Fokker-Planck equation. This step, which is difficult, relies on a key concentration inequality, which we prove in section \ref{subsec.ProofLemma}.

\section{The proof of  Theorem \ref{thm.main1} without common noise} \label{sec.withoutCN}

We assume that $a_0=0$ and, throughout the proof,  we use the fact that $\mathcal V^N$ is the unique solution of the uniformly parabolic backward PDE \eqref{eq.HJBN} and, therefore, is smooth.

\subsection{Some regularity estimates}\label{subsec.regu}
We first establish the, uniform in $N$, regularity estimates  for $\mathcal V^N$.

\begin{lem}\label{lem.estiVN} Assume \eqref{ass.main}. There exists a constant $C>0$ such that, for any $N\geq 1$,
$
\|\mathcal V^N\|_\infty+N \sup_j \|D_{x^j}\mathcal V^N\|_\infty +\|\partial_t \mathcal V^N\|_\infty \leq C.
$
\end{lem}

\begin{rmk}\label{rmk.bddcontrol} The estimate on $D_{x^j}\mathcal V^N$ implies that the optimal feedback of the problem, given by $\alpha^k(t,x) = -D_pH(x^i, ND_{x^j}\mathcal V^N(t,\bx))$ remains uniformly bounded.
\end{rmk}

\begin{proof} The bound on $\mathcal V^N$ is obvious.

We note that $w^i=D_{x^i}\mathcal V^N$ satisfies
\be\label{zaoelsfdjgnc}
\left\{\begin{array}{l}
\ds -\partial_t w^i(t,{\bf x}) -\sum_{k=1}^N \Delta_{x^k}w^i(t,{\bf x}) +\frac1N D_xH(x^i, ND_{x^{i}} \mathcal{V}^N(t,{\bf x}))\\
\ds  +\sum_{k=1}^N D_pH(x^k, N D_{x^k}\mathcal V^N(t, {\bf x}))\cdot D_{x^k}w^i(t,{\bf x})\\
\ds \qquad =\frac1ND_m\mathcal F(m^N_{{\bf x}},x^i)\ \ \  {\rm in}\ \  (0,T)\times (\R^d)^N, \\
\ds  w^i(T,{\bf x})= \frac1N D_m\mathcal G(m^N_{{\bf x}},x^i)\ \ {\rm in} \ \  (\R^d)^N,
\end{array}\right.
\ee
and observe that the maximum principle for linear parabolic equations (see e.g. Theorem 8.1.4 of \cite{Kry96}) together with the condition $|D_x H(x,p)| \leq C(1 + |p|)$ from \eqref{hyp.quad} gives 
\begin{align*}
|w^i(t,\bf{x})| &\leq \int_t^T \big( \frac{1}{N} \norm{D_x H(\cdot, N D_{x^i} \mathcal{V}(s,\cdot)^N)}_{\linf} + \frac{1}{N} \frac{\norm{D_m \sF}_{\linf}}{N} \big) ds + \frac{\norm{D_m G}_{\linf}}{N} \\
&\leq \frac{C}{N} + C \int_t^T \norm{w^i(s,\cdot)}_{\linf}ds
\end{align*}
Taking a supremum in $\bf{x}$ and then applying Gronwall's inequality gives 
$$|D_{x^i}\mathcal{V}^N(t,{\bf x})| \leq \frac{C}{N},$$ 
as required.

 Similarly $w^t= \partial_t \mathcal V^N$ satisfies
\be\label{zaoelsfdjgncTT}
\left\{\begin{array}{l}
\ds -\partial_t w^t(t,{\bf x}) -\sum_{k=1}^N \Delta_{x^k}w^t(t,{\bf x}) +\\
\ds \quad +\sum_{k=1}^N D_pH(x^k, N D_{x^k}\mathcal V^N(t, {\bf x}))\cdot D_{x^k}w^t(t,{\bf x})=0\ \  {\rm in}\ \ (0,T)\times (\R^d)^N, \\
\ds  w^t(T,{\bf x})= - \frac{1}{N}\sum_{k=1}^N {\rm tr} \left[D^2_{y,m}\mathcal G(m^N_{\bx}, x^k)+\frac{1}{N} [D^2_{m,m}\mathcal G(m^N_{\bx}, x^k,x^k)\right]\\
\ds \quad +\frac{1}{N} \sum_k H(x^k, D_m\mathcal G(m^N_{\bx}, x^k) -\mathcal F(m^N_{\bx}) \ \ {\rm in}\ \ (\R^d)^N,
\end{array}\right.
\ee
and the uniform bound on $\|\partial_t \mathcal V^N\|_\infty $ this time follows directly from the maximum principle.
\end{proof}

\begin{lem}\label{It'^o} Assume \eqref{ass.main}.  There is $C>0$ such that, for all
$t_0,s_0 \in [0,T]$ and  $m_0,\bar{m}_0 \in \sP_2(\R^d)$,
  $  |\mathcal U(t_0,m_0) - \mathcal U(s_0,\bar{m}_0)| \leq C\Big(|t_0 - s_0|^{1/2} + {\bf d}_1(m,\bar{m})\Big).$
Moreover,  if $(t_0,m_0)\in [0,T]\times \sP_1(\R^d)$ and $(m,\alpha)$ is optimal in the definition of $\mathcal U(t_0,m_0)$ in \eqref{def.U}, then
$
\|\alpha\|_\infty\leq C.
$
\end{lem}

\begin{proof} The result is standard so we only sketch the argument and refer to \cite{BrCa} and \cite{daudin21} for more details. Fix $(t_0,\bar m_0)\in [0,T]\times \Pk$. It follows from  \eqref{ass.main} that there exists at least a pair $(m, \alpha)$ optimal in the definition of $\mathcal U(t_0,\bar m_0)$. Moreover, for such optimal pair $(m,\alpha)$, there exists a map $ u \in \mathcal{C}^{1,2}_b((t_0,T)\times \R^d)$ with $\alpha_t(x)= -D_pH(x,Du(t,x))$ and such that $(u,m)$ solves the system
$$
\left\{ \begin{array}{l}
\ds -\partial_tu(t,x) -\Delta u(t,x) +H(x,Du(t,x))=\frac{\delta \mathcal F}{\delta m}(m_t,x) \ \  {\rm in}\ \ (t_0,T)\times \R^d\\[1.2mm]
\ds \partial_tm_t(x) -\Delta m_t(x) -{\rm div}(D_pH(x,Du(t,x))m_t(s))=0 \ \  {\rm in}\ \  (t_0,T)\times \R^d,\\[1.2mm]
\ds m_{t_0}=\bar m_0, \; u(T,x)= \frac{\delta \mathcal G}{\delta m}(m_T, x) \ \  {\rm in}\ \ \R^d.
\end{array}\right.
$$
Arguing as for the Lipschitz estimate in Lemma \ref{lem.estiVN}, one can check that $\|Du\|_\infty\leq C$ for some  constant $C>0$ and, since $\alpha= -D_pH(x,Du)$,   $\|\alpha\|_\infty\leq C$. The standard parabolic regularity theory then implies that $\|D\alpha\|_\infty= \|D[D_pH(\cdot, Du(\cdot,\cdot))]\|_\infty \leq C$.
\vs
Fix  $\bar m_1\in \Pk$ and let $\mu$ be the solution to
$
\partial_t \mu-\Delta \mu +{\rm div}(\mu \alpha)=0  \ \  {\rm in}\ \  (t_0,T)\times \R^d \ \ \text{with} \ \  \mu(t_0)= \bar m_1.
$
It is easy to  check that there exists $C=C(\|D\alpha\|_\infty, T)$ such that
$
\sup_{t\in [t_0,T]} {\bf d}_1(\mu(t), m(t))\leq C{\bf d}_1(\bar m_1,\bar m_0).
$
Thus, for some  $C$ depending on $T$, on the regularity of $L$, $\mathcal F$ and $\mathcal G$ and  $\|D\alpha\|_\infty$,
\begin{align*}
\mathcal U(t_0, \bar m_1)&\,\,\, \leq \int_{t_0}^T (\int_{\R^d} L(x, \alpha_t(x))\mu(t,dx)+ \mathcal F(\mu(t)))dt+ \mathcal G(\mu(T))\\
& \,\,
 \leq\,\, \int_{t_0}^T\,\, [\int_{\R^d}\,\,\, L(x, \alpha_t(x))m(t,dx)+ \,\mathcal F(m(t))]dt +\, \mathcal G(m(T))\,+\, C\,\,\sup_{t\in [t_0,T]}\, {\bf d}_1(\mu(t), m(t))\\
& \,\,\leq\, \mathcal U(t_0,\bar m_0) +  C{\bf d}_1(\bar m_1,\bar m_0).
\end{align*}
This establishes the estimate
\begin{align} \label{spacereg}
    | \mathcal{U}(t_0,m_0) - \mathcal{U}(t_0,\bar{m}_0)| \leq C {\bf d}_1(m_0, \bar{m}_0).
\end{align}
Finally, we fix $s_0 < t_0$, and we choose $(m,\alpha)$ optimal in the definition of $\mathcal{U}(s_0,m_0)$. By the dynamic programming (Proposition \ref{pro.dppnocommon}), we have
\newline
  $  \mathcal{U}(s_0,m_0) = \int_{s_0}^{t_0} \Big(\int_{\R^d} L(x, \alpha(t,x)) m_t(dx) + \sF(m_t) \Big) dt + \mathcal{U}(t_0,m_{t_0})$,
and, thus,
\vskip-.175in
\begin{align*}
&    |\mathcal{U}(s_0,m_0) - \mathcal{U}(t_0,m_0)| \leq
    |\int_{s_0}^{t_0}  \Big(\int_{\R^d} L(x, \alpha(t,x)) m_t(dx) + \sF(m_t) \Big) dt | \\
&    + |\mathcal{U}(t_0,m_{t_0}) - \mathcal{U}(t_0,m_0)|
    \leq C (t_0 - s_0) + C {\bf d_1}(m_{t_0}, m_0) \\
    & \qquad \leq C(t_0 - s_0) + C (t_0 - s_0)^{1/2},\\
\end{align*}
\vskip-.1in
where we have used \eqref{spacereg} and the boundedness of $\alpha$, together with the fact that \eqref{convex} implies a similar inequality for $L$. This completes the proof.
\end{proof}
The key estimate on $\mathcal V^N$ is discussed next.

\begin{lem}\label{lem.semiconcesti} Assume \eqref{ass.main}.  There exists an independent of $N$ constant $C$, such that, for any $N\geq 1$,
$\xi=(\xi^i)\in (\R^d)^N$ and $\xi^0\in \R$,
\vskip-.2in
\be\label{estisemiconc}
\begin{split}
\sum_{i,j=1}^ND^2_{x^ix^j}\mathcal V^N(t,{\bf x})\xi^i\cdot \xi^j  + &2 \sum_{i=1}^N D^2_{x^it}\mathcal V^N(t,\bx)\cdot \xi^i\xi^0+ D^2_{tt}\mathcal V^N (t,\bx)(\xi^0)^2\\
&\leq \frac{C}{N}\sum_{i=1}^N |\xi^i|^2+ C(\xi^0)^2.
\end{split}
\ee
\end{lem}

\begin{rmk} Inequality \eqref{estisemiconc} plays a crucial role in the proof of Lemma \ref{lem.semiconcTD} below.  Since $\mathcal V^N$ converges to $\mathcal U$, it follows  that \eqref{estisemiconc} implies the semi-concavity of the extension $\tilde{\mathcal U}:[0,T]\times L^2((\tilde{\Omega}, \tilde{\mathcal{F}},\tilde{\mathbb{P}}); \R^d)\to \R$ defined,  for $X\in L^2(\tilde{\Omega}, \R^d),$ by
$
\tilde{\mathcal U}(t,X):= \mathcal U(t,\mathcal L(X)), 
$
where $(\tilde{\Omega}, \tilde{\F}, \tilde{\bP})$ is a fixed atomless probability space and $\mathcal L(X)$ is the law of the random variable $X$.
\end{rmk}
\vskip-.2in

\begin{proof}
For $1\leq i,j,k\leq N$, let
\vskip-.2in
\begin{equation*}
\begin{array}{ll}
&
\omega^{i}= D_{x^{i}}\mathcal{V}^N\cdot\xi^{i},  \ \   \omega^{i,j}= D^2_{x^{i}x^j}\mathcal{V}^N\xi^{i}\cdot\xi^{j}, \ \  \omega^0= \partial_t\mathcal{V}^N\xi_0, \ \ \omega^{0,0}=\partial_{tt}\mathcal{V}^N(\xi^0)^2 \\
& \omega^{0,i}=\omega^{i,0}= \partial_t D_{x^{i}}\mathcal{V}^N\cdot\xi^0\xi^{i} 
\ \  \tilde{\omega}= \sum_{i,j=0}^N\omega^{i,j} \ \ \text{and} \ \  \sigma_k = \sum_{i=0}^ND_{x^k}\omega^{i}.
\end{array}
\end{equation*}


A straightforward computation gives
\begin{align*}
    & -\partial_t \tilde{\omega} - \sum_{k=1}^N \Delta_{x^k}\tilde{\omega} + \sum_{k=1}^N D_{x^k}\tilde{\omega}.D_pH(x^k, ND_{x^k}\mathcal{V}^N(t,{\bf x})) \\
    & = -N\,\sum_{k=1}^N\, D^2_{pp}H(x^k, ND_{x^k}\mathcal{V}^N(t,{\bf x}))\sigma_k\cdot \sigma_k -2 \,\sum_{k=1}^N\, D^2_{xp}H (x^k, ND_{x^k}\mathcal{V}^N(t,{\bf x}))\xi^k.\sigma_k \\
    & \,\,  -\frac{1}{N}\, \sum_{i=1}^N\, D^2_{xx}H(x^{i},nD_{x^{i}}\mathcal{V}^N(t,{\bf x}))\xi^{i}.\xi^{i}
     \\
    & +\frac{1}{N^2} \, \sum_{i,j=1}^N \, D^2_{mm}\mathcal{F}(m_{\bf x}^N,x^{i},x^j)\xi^{i}.\xi^j + \frac{1}{N}\, \sum_{i=1}^N\,  D_yD_m\mathcal{F}(m_{\bf x}^N,x^{i})\xi^{i}.\xi^{i}
\end{align*}

Denote by $\gamma$ the right-hand-side of the equality above. Recalling that $H$ is strictly convex in the $p$ variable and that $N\partial_{x^k}\mathcal V^N$ is bounded, we have, for all $1\leq k \leq N$,
   $-N D^2_{pp}H \sigma_k \cdot \sigma_k -2 D^2_{xp}H \xi^k.\sigma_k  \leq \frac{C}{N} |\xi^k|^2.$
We can use again the Lipschitz bounds on $\mathcal{V}^N$ and \eqref{hyp.growth2}
to deduce that
$ \gamma(t, {\bf x}) \leq \frac{C}{N} \sum_{k=1}^N |\xi^{i}|^2. $
Next,  fix $(t_0,{\bf x}_0)$ and consider the weak solution $m^N$ to
\hskip-.2in
\begin{equation*}
\begin{split}
& \partial_t m^N(t,{\bf x}) -\sum_{k=1}^N \Delta_{x^k}m^N(t,{\bf x}) \\
&  -\sum_{k=1}^N {\rm div}(D_pH(x^k, N D_{x^k}\mathcal V^N(t, {\bf x}))m^N) = 0 \ \
  {\rm in} \ \ (t_0,T)\times (\R^d)^N, \\ 
&  m^N(t_0,\cdot)= \delta_{\bx_0} \ \ {\rm in} \ \ (\R^d)^N.
\end{split}
\end{equation*}


Integrating the $\tilde \omega-$equation  against $m^N$, we find that, for all $(t_0, {\bf x}_0) \in [0,T] \times (\R^d)^N$,
$
 \tilde{\omega}(t_0, {\bf x_0}) \leq \sup_{ {\bf x}} \|\tilde{\omega}(T, {\bf x}) \|_{\infty} + \frac{C}{N}\sum_{k=1}^N |\xi^k|^2 .
 $
 In order to bound the right-hand side of the inequality above,  we first  note that, by the equation satisfied by $\mathcal V^N$, we have
 \newline
 $
 \partial_t\mathcal V^N(T,\bx)= -\sum_{k=1}^N \Delta_{x^k} \mathcal G^N(\bx) +\frac{1}{N} \sum_{k=1}^N H(x^k, ND_{x^k} \mathcal G^N(\bx)) -\mathcal F^N(\bx),
 $
 where 
 \newline
 $\mathcal F^N(\bx):= \mathcal F(m^N_{\bx})$ and $\mathcal G^N(\bx):= \mathcal G(m^N_{\bx})$, and, similarly, 
 \newline
 $ \partial^2_{tt}\mathcal V^N(T,\bx)= -\,\,\sum_{k=1}^N \,\Delta_{x^k}  \partial_t\mathcal V^N(T,\bx) +\sum_{k=1}^N D_pH(x^k, ND_{x^k} \mathcal G^N(\bx))\cdot D_{x^k}  \partial_t\mathcal V^N(T,\bx).
 $
Recalling the expressions of the derivatives of $\mathcal F^N$ and $\mathcal G^N$ in function of the derivatives of $\mathcal F$ and $\mathcal G$ in Proposition 5.35 of \cite{CaDeBook}, we find, after a tedious but straightforward computation that, under our standing assumptions on $\mathcal F$ and $\mathcal G$, for some $C$,
$\sup_{ {\bf x}} \|\tilde{\omega}(T, {\bf x}) \|_{\infty} \leq \frac{C}{N}\sum_{i=1}^N |\xi^i|^2+ C(\xi^0)^2.$
\end{proof}

\subsection{The easy estimate} \label{subsec.easy}
The second step in the proof of Theorem \ref{thm.main1} is an upper bound of $\mathcal V^N$ in terms of $\mathcal U$. Our strategy will be to first compare $\mathcal U$ to $\hat{\mathcal{V}}^N$, where
\begin{align} \label{lifteddef}
\hat{\mathcal V}^N(t,m):= \int_{(\R^d)^N} \mathcal V^N(t, {\bf x}) \prod_{j=1}^N m(dx^j).
\end{align}
We start with a Lemma, whose proof is a straightforward computation which is essentially the same as the one carried out in the proof of Proposition 3.1 in Cardaliaguet and Masoero \cite{CaMa20}. Hence, we omit the details.

\begin{lem} \label{lem.subsol}
Let $\hat{\mathcal{V}}^N$ be given by \eqref{lifteddef}. Then $\hat{\mathcal V}^N$ is smooth and satisfies the inequality
$$
\left\{\begin{array}{l}
\ds -\partial_t  \hat{\mathcal V}^N(t,m)-\int_{\R^d}{\rm div}( D_m \hat{\mathcal V}^N(t,m,y))m(dy)\\
\ds  +\int_{\R^d} H(y, D_m\hat{\mathcal V}^N (t,m,y))m(dy)\leq \hat{\mathcal F}^N(m) \ \
{\rm in}\ \  (0,T)\times \Pk,\\
\ds \hat{\mathcal V}^N(T,m)= \hat{\mathcal G}^N(m)\ \  {\rm in} \ \  \Pk,
\end{array}\right.
$$
where
$
\hat{\mathcal F}^N(m)= \,\int_{(\R^d)^N}\, \mathcal F( m^N_{{\bf x}}) \,\prod_{j=1}^N \, m(dx^j)$
and
 $\hat{\mathcal G}^N(m)= \,\int_{(\R^d)^N}\, \mathcal G( m^N_{{\bf x}}) \prod_{j=1}^N m(dx^j).
$
\end{lem}

Next we prove the easier inequality in Theorem \eqref{thm.main1}.

\begin{prop}\label{lem.ineqeasy}
There exist constants $C$ depending on the data and $\beta$ depending only on $d$ such that, for all $(t,\bx_0)\in  [0,T]\times (\R^d)^N$,
\be\label{ineq1}
{\mathcal V}^N(t,m^N_{\bx_0})\leq \mathcal U(t,m^N_{\bx_0})+\dfrac{C}{N^\beta} (1+M_2^{1/2}(m^N_{\bx_0})).
\ee
\end{prop}

\begin{proof} Theorem 1 in \cite{FoGu} gives constants $C$ and $\beta$ depending only on $d$ such that, for any $m\in \sP_2(\R^d)$ and for all $N \in \N$,
$
\int_{(\R^d)^N} {\bf d}_1(m^N_{{\bf x}},m) \prod_{i=1}^N m(dx^i)\leq \dfrac{C}{N^\beta} M_2^{1/2}(m).
$
\vs
Fix  $(t_0,m_0)\in [0,T)\times \mathcal P_2(\R^d)$ and let $\alpha^*$ be optimal in the definition of $\mathcal U(t_0,m_0)$.
Using Lemma \eqref{lem.subsol} together with a standard verification argument, for example,  using It\^o's formula in Theorem 5.99 of \cite{CaDeBook}, we see that
$$
\hat{\mathcal{V}}^N(t_0,m_0) \leq  \inf_{\alpha \in \sA(t_0,m_0)} \Big\{ \int_{t_0}^T  \big( \int_{\R^d} L(x, \alpha(t,x)) m_t(dx) + \hat{\mathcal F}^N(m_t) \Big) dt + \hat{\mathcal G}^N(m_T) \Big\}
$$
and, hence,
\be \label{vhatest}
\hat{\mathcal{V}}^N(t_0,m_0) \leq \int_{t_0}^T \big(  \int_{\R^d}L(x, \alpha^*(t,x)) + \hat{\mathcal F}^N(m_t) \Big) dt + \hat{\mathcal G}^N(m_T).
\ee
Since, in view of  Lemma~\ref{It'^o}, $\alpha^*$ is uniformly bounded by a constant independent of $N$,
 an easy computation shows that the corresponding state process satisfies
 \newline
$
\sup_{t\in [t_0,T]}\int_{\R^d} |x|^2 m(t,dx) \leq (1+CT)\int_{\R^d} |x|^2m_0(dx) +CT.
$
It then follows from  the Lipschitz continuity of $\mathcal F$  with respect to ${\bf d_1}$  that

\begin{align*}
\left| \hat{\mathcal{F}}^N(m(t)) - \mathcal{F}(m(t)) \right| & \leq C \int_{{(\R^d)}^N}{\bf d}_1(m^N_{{\bf x}},m(t))  \prod_{j=1}^N m(t,dx^j) \\
& \leq  \dfrac{C}{N^{\beta}}(1+M_2^{1/2}(m_0))
\end{align*}
and, similarly $\displaystyle \left| \hat{\mathcal{G}}^N(m(T)) - \mathcal{G}(m(T)) \right| \leq  \dfrac{C}{N^{\beta}}(1+M_2^{1/2}(m_0)) $.

Using the optimality of $\alpha^*$, \eqref{vhatest} and the estimates above 
\begin{align*}
\hat{\mathcal V}^N(t_0,m_0)& \leq  \E[\int_{t_0}^T \big( L(X_t, \alpha^*_t) + {\mathcal F}(\sL(X_t)) \Big) dt + {\mathcal G}(\sL(X_T))]+    \dfrac{C}{N^{\beta}}(1+M^{1/2}_2(m_0))\\
&\leq \mathcal U(t_0,m_0) + \dfrac{C}{N^{\beta}}(1+M_2^{1/2}(m_0)).
\end{align*}
Fix now $\bx_0\in (\R^d)^N$. Then the Lipschitz estimate on $\mathcal V^N$ and the same argument as above yield
$
\abs{\mathcal{V}^N(t_0,\bx_0) - \hat{\mathcal{V}}^N(t_0,m^N_{\bx_0})} \leq  \dfrac{C}{N^{\beta}}(1+M^{1/2}_2(m^N_{\bx_0})).
$
Putting together the last two estimates gives \eqref{ineq1}.
\end{proof}

\subsection{The main estimate} \label{subsec.main}

The aim of this section is to prove the opposite inequality.

\begin{prop}\label{prop.hardineq} Assume \eqref{ass.main}. There exists  $\beta\in (0,1]$ depending  only on the  dimension and  $C>0$ depending  on the data,  such that, for any $N\geq1$ and any $(t,\bx)\in [0,T]\times (\R^d)^N$,
\be \label{harderineq}
\mathcal U(t,m^N_{{\bf x}}) -{\mathcal V}^N(t,{\bf x}) \leq  \dfrac{C}{N^{\beta}}(1+ \frac{1}{N} \sum_{i=1}^N |x^i|^2).
\ee
\end{prop}

As pointed out in the introduction, the main difficulty is that it  does not
seem possible, at least to us, how to transform an optimal control for the $\mathcal V^N$ which depends on each particle into a feedback for $\mathcal U$. We overcome this difficulty by dividing the players into subgroups in such a way that the optimal controls for the agents in each subgroup are close and showing a propagation of chaos-type result  for each subgroup using a concentration inequality.

We begin explaining how to create the subgroups based on an appropriate partition of $\{1,\ldots,N\}$. 


\begin{lem}\label{takis40}
For each $\delta>0$ there exist a constant $C$ depending only on the data ($\mathcal{F}, \mathcal{G},H)$), a partition $(C^j)_{j\in \{1, \dots, J\}}$ of $\{1, \dots, N\}$ such that $J \leq C \delta^{-d}$ and,  for $j=1, \dots, J$, controls $\bar \alpha^j\in \R^d$ such that, for all $k\in C^j$,
\be\label{def.alphaj} \left| H(x^k_0, ND_{x^k}\mathcal V^N(t_0,\bx_0)) + \bar \alpha^j \cdot (ND_{x^k}\mathcal V^N(t_0,\bx_0))+L(x^k_0 , \bar \alpha^j)\right| \leq C\delta.
\ee
\end{lem}

\begin{proof}
%
Let $\hat{\alpha}^k(t,\bx) =  -D_p H(x^k,ND_{x^k} \mathcal{V}^N(t,\bx))$ be  the optimal feedback for particle $k$, and  recall (see Remark \ref{rmk.bddcontrol}), that there exists $R$ depending only on the data such that $|\hat{\alpha}^k(t,\bx)| \leq R$.

Given $\delta > 0 $, we can find a $\delta$-covering of  $B_R \subset \R^d$  consisting of $J \leq C\delta^{-d}$ balls of radius $\delta$ centered at  $(\bar{\alpha}_j)_{j \in \{1,...,J\}} \subset B_R$.

Then, we choose the partition $(C^j)_{j \in 1,\dots J}$ so that, for each $k \in C^j$, $|\hat{\alpha}^k(t,\bx) - \bar \alpha^j| \leq \delta$. It follows using \eqref{hyp.l} that,  for each $k \in C^j$,
\begin{align*}
    | H(x^k_0, ND_{x^k}\mathcal V^N(t_0,\bx_0)) + \bar \alpha^j \cdot (ND_{x^k}\mathcal V^N(t_0,\bx_0))+L(x^k_0 , \bar \alpha^j)| \\ =  | \big(\alpha^j - \hat{\alpha}^k(t_0,\bx_0) \big) \cdot (ND_{x^k}\mathcal V^N(t_0,\bx_0))+L(x^k_0 , \bar \alpha^j) - L(x_0^k, \hat{\alpha}(t_0,\bx_0)) | \\
    \leq \big(ND_{x^k}\mathcal V^N(t_0,\bx_0)) + \norm{D_a L}_{\linf(\R^d \times B_R)}\big)|\hat{\alpha}^k(t_0,\bx_0) - \bar \alpha^j| \leq C \delta.
\end{align*}
\end{proof}

For $\delta >0$ we consider such a partition $(C^j)_{j \in \{1,\dots ,J \}}$ of $\{1,\dots,N \}$ with associated controls $\bar{\alpha}^1,\dots \bar{\alpha}^J$ satisfying the conditions of Lemma \ref{takis40} and we define $n^j \doteq |C^j|$ for all $j \in \{ 1,\dots,J \}$. Fix $j\in \{1,\ldots,J\}$, set $\alpha^k=\bar \alpha^j$ if $k\in C^j$, let, for $t_0,s_0 \in [0,T]$ and $\bf{x}_0, \bf{y}_0 \in (\R^d)^N$,

\begin{equation}
\begin{array}{ll}
& \displaystyle X^k_{t_0+\tau}= x^k_0+\tau \alpha^k + \sqrt{2}B^k_\tau \ \ \text{and} \ \ \displaystyle Y^k_{s_0+\tau}= y^k_0+\tau \alpha^k + \sqrt{2} B^k_{\tau},
\\
& \displaystyle m^j_{{\bf X}_{t_0+\tau}}=\dfrac{1}{n^j}\sum_{k\in C^j}  \delta_{X^k_{t_0+\tau}}
 \ \ \text{and} \ \ \displaystyle m^j_{{\bf Y}_{s_0+\tau}}= \dfrac{1}{n^j}\sum_{k\in C^j} \delta_{Y^k_{s_0+\tau}},
 \end{array}
  \label{ConstructionXetY}
\end{equation}

consider the solution $m^j$ to
\begin{equation}
\newline
 \partial_t m^j- \Delta m^j +\bar\alpha^j \cdot Dm^j=0 \ \  {\rm in} \ \ (s_0,T)\times \R^d \ \ \text{and} \ \
m^j(s_0,\cdot)= m^j_{\by_0} \ \ \text{in} \ \ \R^d,
\label{Constructionmetmj}
\end{equation}
and, finally, set $m(s)= \dfrac{1}{N} \sum_{j\in J} n^jm^j(s)$.

We state next the concentration inequality we need for the proof of Proposition \ref{prop.hardineq}.


\begin{lem}\label{lem.estid1}  There exist a positive constant $\beta\in (0,1/2)$, depending on $d$ and a positive constant $C$, which
 depends only on $\sup_j |\bar\alpha^j|$,  $d$ and $T$, such that, for all $h\geq 0$,
\be\label{ineq.concentration}
\E\left[  {\bf d}_1( m^j(s_0+h), m^j_{{\bf Y}_{s_0+h}}) \right] \leq  C(1+M_2^{1/2}(m^j(s_0)))(h/n^j)^\beta,
\ee
\begin{equation}
\E\left[  {\bf d}_1( m^j(s_0+h), m^j_{{\bf X}_{t_0+h}}) \right] \leq  {(n^j)}^{-1} \sum_{k\in C^j}|x^k_0-y^k_0|+ C(1+M_2^{1/2}(m^j(s_0)))(h/n^j)^\beta, 
\label{LemmaEstid1ineq2}
\end{equation}
and, as a consequence,
\begin{equation}
\E\left[ {\bf d}_1( m(s_0+h), m^N_{{\bf Y}_{s_0+h}}) \right] \leq C\delta^{-d\beta}(1+ M_2(m(s_0))^{\frac{1}{2}})(h/N)^\beta
\label{LemmaEstid1ineq3}
\end{equation}

\begin{equation}
\E\left[ {\bf d}_1( m(s_0+h), m^N_{{\bf X}_{t_0+h}}) \right] \leq  \frac{1}{N}\sum_{k=1}^N |x_0^k-y_0^k|+ C\delta^{-d\beta}(1+ M_2(m(s_0))^{\frac{1}{2}}) (h/N)^\beta. 
\label{LemmaEstid1ineq4}
\end{equation}

\end{lem}



\begin{proof}

Inequality \eqref{ineq.concentration} is precisely the concentration inequality \eqref{ConcentrationIneqGeneralFormulation} that we treat separately in Proposition \ref{ConcentrationInequalityGeneralCase} of Section \ref{subsec.ProofLemma} because it is interesting in its own. Being $Y_{s_0+h}^k - X^k_{t_0+h} = y_0^k - x_0^k$ for all $h\geq 0$ and all $k \in \{ 1,\dots N \}$, inequality \eqref{LemmaEstid1ineq2} follows in a straightforward way from \eqref{ineq.concentration}. Similarly \eqref{LemmaEstid1ineq4} follows from \eqref{LemmaEstid1ineq3}. It remains to prove estimate \eqref{LemmaEstid1ineq3}.

Using \eqref{ineq.concentration} as well as the Cauchy-Schwarz inequality,  the concavity of the maps $n \rightarrow n^{1-\beta}$ and $n\rightarrow n^{1-2\beta}$, the fact that $\sum_j n^j=N$, and the assumption that $\beta\in (0,1/2)$, we obtain the following string of inequalities
\begin{align*}
\E\left[ {\bf d}_1( m(s_0+h), m^N_{{\bf Y}_{s_0+h}}) \right] & \leq \sum_{j\in J} \frac{n^j}{N} \E\left[  {\bf d}_1( m^j(s_0+h), m^j_{{\bf Y}_{s_0+h}}) \right] \\
& \leq  C\sum_{j\in J} \frac{n^j}{N} (1+M_2^{1/2}(m^j(s_0)))\frac{h^{\beta} }{(n^j)^\beta} \\
& \hskip-.6in \leq Ch^{\beta} \big[\sum_{j\in J} \frac{(n^j)^{1-\beta}}{N}  +
( \sum_{j\in J} \frac{n^j}{N}M_2(m^j(s_0)) )^{1/2}(\sum_{j\in J} \frac{n^j}{N (n^j)^{2\beta}} )^{1/2}\big]\\
&  \hskip-.6in \leq Ch^{\beta} \big[\frac{J}{N} \left( \sum_{j \in J} \frac{n^j}{J}\right)^{1-\beta}+ M_2^{1/2}(m(s_0)) \sqrt{\frac{J}{N}}( \sum_{j \in J} \frac{1}{J}(n^j)^{1-2\beta})^{1/2}\big] \\
&  \hskip-.6in\leq C {(\dfrac{J h}{N}})^\beta \big( 1 + M_2^{1/2}(m(s_0)) \Big).
\end{align*}
Recalling that $J\leq C\delta^{-d}$ is enough to conclude.

\end{proof}

We are now ready, using the above construction, to prove Proposition \ref{prop.hardineq}.

\begin{proof}[Proof of Proposition \ref{prop.hardineq}] Following a viscosity solutions-type argument, we double the variables and, for $\theta,\lambda\in (0,1)$, we set
\be
\begin{split}
M:= & \max_{(t,{\bf x}), (s,{\bf y})\in [0,T]\times (\R^d)^N} e^{ s} (\mathcal U(s,m^N_{{\bf y}}) -{\mathcal V}^N(t,{\bf x}))\\
& -\frac{1}{2\theta N} \sum_{i=1}^N |x^i-y^i|^2  -\frac{1}{2\theta }|s-t|^2 -\frac{\lambda}{2N} \sum_{i=1}^N |y^i|^2.
\end{split}
\label{DefinitionofM25Nov}
\ee
We denote by $((t_0,\bx_0),(s_0, \by_0))$ a maximum point in the expression above. Using the uniform bound on $\mathcal U$ and $\mathcal V^N$ and the Lipschitz estimate for $\mathcal V^N$ we can estimate the error related to the penalization. We find that there exists $C>0$ such that,

\begin{equation}
\frac{1}{N} \sum_{i=1}^N |x_0^{i}-y_0^{i}|^2 +|s_0-t_0|^2 \leq C \theta^2 \ \ \text{and} \ \  \frac{1}{N} \sum_{i=1}^N |y^i_0|^2\leq \dfrac{C}{\lambda}.
\label{lem.estixkyk}
\end{equation}





\vs
Now we fix $\delta >0$ and we define $(\textbf{X}_t)_{t \geq t_0}$, $(\textbf{Y}_s)_{s \geq s_0}$, $m^j$ and $m$ according to \eqref{ConstructionXetY} and \eqref{Constructionmetmj} for some partition $(C^j)_{j \in \{1,\dots ,J \}}$ of $\{1,\dots,N \}$ with associated controls $\bar{\alpha}^1,\dots \bar{\alpha}^J$ satisfying the conditions of Lemma \ref{takis40}. By estimate \eqref{lem.estixkyk} it holds, in particular, $\frac{1}{N} \sum_{i=1}^N |x_0^{i}-y_0^{i}| \leq C\theta$ and $M_2(m(s_0)) \leq C \lambda^{-1}$.



The Lipschitz regularity of $\mathcal U$ in Lemma \ref{It'^o} and the definition of ${\bf X}_t$ and ${\bf Y}_t$ give, by definition of M, 
\begin{align*}
M & \geq \E\Bigl[ e^{s_0+h} (\mathcal U(s_0+h,m^N_{{\bf Y}_{s_0+h}})  - \mathcal V^N(t_0+h,{\bf X}_{t_0+h}))\\
&  \qquad\qquad -\frac{1}{2\theta}\left(\frac{1}{N} \sum_{k=1}^N |Y^k_{s_0+h}-X^k_{t_0+h}|^2
+(t_0-s_0)^2\right)  -\frac{\lambda}{2N} \sum_{i=1}^N |Y^i_{s_0+h}|^2 \Bigr] \\
& \geq \E\Bigr[ e^{s_0+h} ( \mathcal U(s_0+h, m(s_0+h))- \mathcal V^N(t_0+h, {\bf X}_{t_0+h}))\Bigr] - C\delta^{-d\beta}(1+ \lambda^{-\frac{1}{2}})\frac{h^{\beta}}{N^{\beta}} \\
& \qquad\qquad -\frac{1}{2\theta}\left(\frac{1}{N} \sum_{k=1}^N |y^k_0-x^k_0|^2+(s_0-t_0)^2\right) -\frac{\lambda}{2N} \sum_{i=1}^N (|y^i_0|+Ch^{1/2})^2.
\end{align*}
\vs
To continue,  we need a dynamic programming-type argument, which is stated next. Its proof is postponed for later in the paper.
\vs

\begin{lem} \label{lem.PDD} With the notation above, we have
\begin{align*}
\mathcal U(s_0+h,m(s_0+h))& \geq \mathcal U(s_0,m_{{\bf y}^N_0})\\[1.5mm]
&-\int_{s_0}^{s_0+h} (\sum_{j=1}^J \int_{\R^d}\dfrac{1}{N}n^jL(x,\bar\alpha^j) m^j(s,x)dx +\mathcal F(m(s)))ds.
\end{align*}
\end{lem}

Using It\^o's formula for $\mathcal V^N$ we find
\begin{align*}
&M  \geq e^{s_0+h}\mathcal U(s_0,m^N_{{\bf y}_0}) -e^{s_0+h} \int_{s_0}^{s_0+h}(\int_{\R^d}  \sum_{j=1}^J \dfrac{1}{N}n^j L(x, \bar\alpha^j)m^j(s,x)dx+\mathcal F(m(s)))ds \\
&  - e^{s_0+h} \E\left[  \mathcal V^N(t_0,{\bf x}_0)  +\int_{t_0}^{t_0+h}\big(\partial_t + \sum_{k=1}^N [\Delta_{x_k} + {\alpha^k} \cdot D_{x_k}]\big) V^N(t,{\bf X}_t) dt \right] -\\
&C\delta^{-d\beta}(1+ \lambda^{-\frac{1}{2}})\frac{h^{\beta}}{N^{\beta}}  -\frac{1}{2\theta}\big(\frac{1}{N} \sum_{k=1}^N |y^k_0-x^k_0|^2+(s_0-t_0)^2\big)-\frac{\lambda}{2N} \sum_{i=1}^N (|y^i_0|+Ch^{1/2})^2.
\end{align*}
Since the $\bar\alpha^j$ are uniformly bounded, the map $L(\cdot ,\bar\alpha^j)$ is uniformly Lipschitz  independently of $j$.
Hence, using  Lemma \ref{lem.estid1} and Lemma \ref{lem.estixkyk}, we find
\begin{align*}
&\int_{s_0}^{s_0+h}\int_{\R^d}  \sum_{j=1}^J \dfrac{1}{N}n^j L(x, \bar\alpha^j)m^j(s,x)dxds\\
&   \leq \E\left[\int_{s_0}^{s_0+h} \sum_{j=1}^J ( \sum_{k\in C^j}\dfrac{1}{N} L(X^k_{t_0-s_0+s}, \bar\alpha^j)  + C\dfrac{1}{N}n^j {\bf d}_1( m^j(s), m^j_{{\bf X}_{t_0-s_0+s}}) ) ds \right] \\
&   \leq \E\left[\int_{t_0}^{t_0+h}  \sum_{k=1}^N \dfrac{1}{N}L(X^k_s, \alpha^k)ds \right] + C\theta h+ C  \sum_{j=1}^J \dfrac{1}{N}n^j(1+M_2^{1/2}(m^j_{s_0})) \frac{h^{\beta}}{(n^j)^\beta}\\
&   \leq \E\left[\int_{t_0}^{t_0+h}  \sum_{k=1}^N \dfrac{1}{N}L(X^k_{s}, \alpha^k)ds \right] +  C\theta h+C\delta^{-d\beta}(1+ \lambda^{-\frac{1}{2}})\frac{h^{\beta}}{N^{\beta}}.
\end{align*}
Note that  in the last inequality we used exactly the same argument as for the proof  given above for  the third inequality of Lemma \ref{lem.estid1}.
\vs
Hence, recalling the optimality of $(\bx_0,\by_0)$  in \eqref{DefinitionofM25Nov} and employing  the equation for $\mathcal V^N$, we get
\begin{align*}
0 & \geq (e^{s_0+h}-e^{s_0})(\mathcal U(s_0,m^N_{{\bf y}_0})-\mathcal V^N(t_0,{\bf x}_0)) - C\delta^{-d\beta}(1+ \lambda^{-\frac{1}{2}})\frac{h^{\beta}}{N^{\beta}} \\
&-C\lambda h^{1/2} N^{-1} \sum_{i=1}^N|y^i_0|-C\theta h
 -e^{s_0+h}\E\left[ \int_{s_0}^{s_0+h}(\mathcal F(m(s))-\mathcal F(m^N_{{\bf X}_{s_0-t_0+s}}))ds\right]\\
&\qquad  -  e^{s_0+h} \E\big[ \frac{1}{N}\int_{t_0}^{t_0+h} \sum_{k=1}^N ( L(X^k_s, \alpha^k)+ \alpha^k \cdot (ND_{x^k}\mathcal V(s,{\bf X}_s))\\[1.5mm]
& \qquad +H(X^k_s, ND_{x^k}\mathcal V(s,{\bf X}_s))) ds \big].
\end{align*}
\vs
Using  the Lipschitz regularity of $\mathcal F$ and Lemma \ref{lem.estid1} to deal with the difference of the $\mathcal F$ and \eqref{lem.estixkyk} to deal with the term in $\sum_i |y^i_0|$, we find
\begin{align*}
0 & \geq  e^{s_0}h(\mathcal U(s_0,m^N_{{\bf y}_0})-\mathcal V^N(t_0,{\bf x}_0)) - C\delta^{-d\beta}(1+ \lambda^{-\frac{1}{2}})\frac{h^{\beta}}{N^{\beta}}-C\lambda^{1/2}h^{1/2} -C\theta h-Ch^2\\
&\qquad  -  e^{s_0+h} \E\big[ \frac{1}{N}\int_{t_0}^{t_0+h} \sum_{k=1}^N ( L(X^k_s, \alpha^k)+ \alpha^k \cdot (ND_{x^k}\mathcal V(s,{\bf X}_s))\\
 &\qquad +H(X^k_s, ND_{x^k}\mathcal V(s,{\bf X}_s))ds) ds \big].
\end{align*}
The regularity of $L$ and $H$ and the uniform boundedness of the $\alpha^k$ and of $ND_{x^k}\mathcal V^N$ allow to  infer that
\begin{align*}
0 & \geq e^{s_0} h(\mathcal U(s_0,m^N_{{\bf y}_0})-\mathcal V^N(t_0, {\bf x}_0)) - C\delta^{-d\beta}(1+ \lambda^{-\frac{1}{2}})\frac{h^{\beta}}{N^{\beta}} -C\lambda^{1/2}h^{1/2}-C\theta h   \\
& -Ch^2 - e^{s_0+h} \E\big[ \frac{1}{N}\int_{t_0}^{t_0+h} \sum_{k=1}^N ( L(x^k_0, \alpha^k)ds+ \alpha^k \cdot (ND_{x^k}\mathcal V(s,{\bf X}_s))\\
& +H(x^k_0, ND_{x^k}\mathcal V(s,{\bf X}_s))) ds \big] -Ch^{3/2}.
\end{align*}
and, in view of \eqref{def.alphaj}, 
\be\label{IneqProp3.8}
\begin{split}
0 & \geq  e^{s_0} h(\mathcal U(s_0,m^N_{{\bf y}_0})-\mathcal V^N(t_0, {\bf x}_0)) - C\delta^{-d\beta}(1+ \lambda^{-\frac{1}{2}})\frac{h^{\beta}}{N^{\beta}}-C\lambda^{1/2}h^{1/2}\\[2mm]
&-- C  \E\left[ \frac{1}{N}\int_{t_0}^{t_0+h} \sum_{k=1}^N |ND_{x^k}\mathcal V^N(s,{\bf X}_s))-ND_{x^k}\mathcal V^N(s,\bx_0))|ds \right]\\[1.2mm]
& -C\theta h -Ch^{3/2} - Ch\delta.
\end{split}
\ee
The semiconcavity of $\mathcal V^N$ and the penalization by the term in $\theta$ give the next lemma. The proof is postponed to end of the section.

\begin{lem}\label{lem.semiconcTD} For any $(t,\bx)\in [0,T]\times (\R^d)^N$,
\begin{align*}
& \sum_{k=1}^N | D_{x^k}\mathcal V^N(t,\bx)-D_{x^k}\mathcal V^N(t_0,\bx_0)|\\
& \qquad \leq \frac{C}{N}\sum_{k=1}^N |x^k-x^k_0|+  \left( \frac{C}{N\theta} \sum_{k=1}^N (|x^k-x^k_0|+|x^k-x^k_0|^2) \right)^{1/2}+ \frac{C}{\theta^{1/2}}|t-t_0|^{1/2}.
\end{align*}

\end{lem}
\vs
We continue with the ongoing proof. Inserting the estimate of Lemma~\ref{lem.semiconcTD} in \eqref{IneqProp3.8}, we obtain
\begin{align*}
&0  \geq  e^{s_0}h(\mathcal U(s_0,m^N_{{\bf y}_0})-\mathcal V^N(t_0,{\bf x}_0)) - C\big(\delta^{-d\beta}(1+ \lambda^{-\frac{1}{2}})\frac{h^{\beta}}{N^{\beta}} + \lambda^{1/2}h^{1/2}\\
& +(\theta+ \delta) h +h^{3/2}\big)\\[1.5mm]
&  - C  \E \int_{t_0}^t (\frac{1}{N}\sum_{k=1}^N |X^k_s-x^k_0| + \big( \frac{1}{N\theta} \sum_{k=1}^N (|X^k_s-x^k_0|+|X^k_s-x^k_0|^2)\big)^{1/2}\\
& +\frac{C}{\theta^{1/2}}|s-t_0|^{1/2} )ds \\[1.5mm]
& \geq  e^{s_0}h (\mathcal U(s_0,m^N_{{\bf y}_0})-\mathcal V^N(t_0,{\bf x}_0)) -C\delta^{-d\beta}(1+ \lambda^{-\frac{1}{2}})\frac{h^{\beta}}{N^{\beta}}\\[1.5mm]
& \qquad \qquad -C(\theta +\delta)h  -C\lambda^{1/2}h^{1/2} - C\theta^{-1/2}h(h^{1/2}+h )^{1/2}.
\end{align*}
Dividing by $h$ we find, for each choice of $\theta,\lambda, \delta > 0$ and $0 < h \leq (T - s_0) \wedge (T- t_0)$,  
\be\notag
\begin{split}
&e^{s_0}(\mathcal{U}(s_0,m^N_{{\bf y}_0}) - \mathcal{V}^N(t_0,{\bf x}_0)) \leq C\frac{h^{\beta-1}}{N^{\beta}\delta^{d\beta}}(1+ \lambda^{-1/2}) + C(\theta + \delta) + C \lambda^{1/2}h^{-1/2} \\
&\hskip.8in + Ch^{1/4}\theta^{-1/2}.
\end{split}
\ee

We take $\theta = h^{\alpha_1},  \ \ \delta = (\frac{\lambda^{-1/2}h^{\beta -1}}{N^{\beta}})^{\alpha_2},  \ \ \lambda = N^{-\alpha_3} \ \ \text{ and} \ \  h= N^{-\alpha_4}.$

Making  appropriate choices of $\alpha_1$, $\alpha_2$, $\alpha_3$ and $\alpha_4$ we deduce 
\begin{align} \label{lastest}
    e^{s_0}(\mathcal{U}(s_0,m^N_{{\bf y}_0}) - \mathcal{V}^N(t_0,{\bf x}_0)) \leq CN^{-\tilde{\beta}}
\end{align}
for some $\tilde{\beta} = \tilde{\beta}(\beta) \in (0,1/2)$ and for  $N$ such that $h = N^{-\alpha_4} \leq (T- s_0) \wedge (T-t_0)$.

For $N$ such that $h = N^{-\alpha_4} \geq (T- s_0) \wedge (T-t_0)$, we have by \eqref{lem.estixkyk} that $(T- s_0) \vee (T -s_0) \leq h + C\theta$,
and, so, using Lemma \ref{lem.estiVN} and Lemma \ref{It'^o}, we find
\begin{align*}
  & |\mathcal{U}(s_0,m^N_{{\bf y}_0}) - \mathcal{V}^N(t_0,{\bf x}_0)| \leq |\mathcal{U}(s_0,m^N_{{\bf y}_0}) - \mathcal{G}(m^N_{{\bf y}_0})| + |\mathcal{G}(m^N_{{\bf y}_0}) - \mathcal{G}(m^N_{{\bf x}_0})| \\
   &+ | \mathcal{G}(m^N_{{\bf x}_0}) - \mathcal{V}^N(t_0,{\bf x}_0)|
   \leq C (h + \theta)^{1/2} + C \theta + C (h + \theta) \leq C N^{-\tilde{\beta}},
\end{align*}
where in the last line we choose $\tilde{\beta}$ even smaller if necessary. With this choice of $\tilde{\beta}$, we have now established that \eqref{lastest} holds for all values of $N$.

Finally, we conclude that, for all $(t,{\bf x}) \in [0,T] \times (\R^d)^N$,
\begin{align*}
e^t (\mathcal{U}(t,m^N_{{\bf x}}) - \mathcal{V}^N(t,{\bf x})) &\leq e^{s_0}(\mathcal{U}(s_0,m^N_{{\bf y}_0}) - \mathcal{V}^N(t_0,{\bf x}_0)) + \frac{\lambda}{2N} \sum_{i=1}^N|x^{i}|^2 \\
& \leq C N^{-\min (\tilde{\beta},\alpha_3)}(1 + \frac{1}{N} \sum_{i=1}^N|x^{i}|^2).
\end{align*}

\end{proof}

Before proving the various lemmas used in the proof of Proposition \ref{prop.hardineq}, we complete the proof of the main result.

\begin{proof}[Proof of Theorem \ref{thm.main1}] Combining Proposition ~\ref{lem.ineqeasy} and  Proposition~\ref{prop.hardineq} we know that there exist $\beta\in (0,1]$ depending on dimension and $C>0$ depending on the data such that, for any $(t,\bx)\in [0,T]\times (\R^d)^N$,
$$
\left| \mathcal U(t,m^N_{{\bf x}}) -{\mathcal V}^N(t,{\bf x})) \right| \leq C N^{-\beta}(1+ M_2^{1/2}(m^N_{\bx})+ M_2(m^N_{\bx})) \leq C N^{-\beta}(1+ M_2(m^N_{\bx})).
$$
\end{proof}
We continue with the proofs of the several auxiliary results sated earlier.
\vs


\begin{proof}[Proof of Lemma \ref{lem.PDD}]
For  $K\in \N$ and  any  nonnegative  integrable functions \\
$m^1_0, \dots, m^K_0$  on $\R^d$ such that $\sum_{k=1}^K m^k_0 \in \sP(\R^d)$, let
\begin{align*}
\mathcal U^K(t_0,m^1_0, \dots, m^K_0) & :=\inf_{(m^1,\beta^1), \dots, (m^K,\beta^K)} \int_{t_0}^{T} ( \int_{\R^d} \sum_{k=1}^K L(x,\frac{\beta^k(t,x)}{m^k(t,dx)})m^k(t,x)dx\\
& +\mathcal F(\sum_{k=1}^K m^k(t)) )dt
 +\mathcal G(\sum_{k=1}^K m^k(T)),
\end{align*}
where the infimum is taken over the tuple of measures $(m^k,\beta^k)$ (the $\beta^k$ being a vector measure) with $\beta^k<<m^k$ such that  $(m^k,\beta^k)$ solve in the sense of distributions,
$$
\partial_t m^k -\Delta m^k +{\rm div}(\beta^k)=0 \ \ \text{in} \ \ (t_0,T]\times \R^d \ \ \text{and} \ \  m^k(t_0)= m^k_0 \ \ \text{in} \ \ \R^d.
$$
We establish next that
$\mathcal U^K(t_0,m^1_0, \dots, m^K_0) = \mathcal U(t_0,m^1_0+ \dots + m^K_0),
$
and the result will then follow from Proposition \ref{pro.dppnocommon}.

Since obviously  $\mathcal U^K(t_0,m^1_0, \dots, m^K_0) \leq \mathcal U(t_0,m^1_0+ \dots + m^K_0)$, next we concentrate on the reverse  inequality.
\vs

Fix $\ep>0$, let $(m^1,\beta^1, \dots, m^K,\beta^K)$ be $\ep-$optimal for $\mathcal U^K(t_0,m^1_0, \dots, m^K_0)$, and set $\beta = \sum_{k=1}^N\beta ^K$ and $m(t)= \sum_{k=1}^N m^k(t)$. Then $(m,\beta)$ solves
$$
\partial_t m -\Delta m +{\rm div}(\beta)=0  \ \ \text{in} \ \ (t_0,T]\times \R^d \ \ \text{and} \ \  m(t_0)= m_0 \ \ \text{in} \ \ \R^d.
$$
and we have
\begin{align*}
& \ep+ \mathcal U^K(t_0,m^1_0, \dots, m^K_0) \\[1.5mm]
& \geq \int_{t_0}^{T} ( \int_{\R^d} \sum_{k=1}^K L(x,\frac{\beta^k(t,x)}{m^k(t,x)})\frac{m^k(t,x)}{m(t,x)}m(t,x)dx+\mathcal F(\sum_{k=1}^K m^k(t)) )dt
+\mathcal G(\sum_{k=1}^K m^k(T)) \\[1.5mm]
& \geq \int_{t_0}^{T} ( \int_{\R^d}  L(x,\frac{\sum_{k=1}^K \beta^k(t,x)}{m(t,x)})m(t,x)dx+\mathcal F(m(t)) )dt +\mathcal G(m(T)) \\[1.5mm]
&\qquad  \geq \mathcal U(t_0,m_0),
\end{align*}
where the second inequality follows from the  convexity of the map $(\beta, m)\to m L(x,\dfrac{\beta}{m})$ and the third one by the definition of $\mathcal U$.
\end{proof}

\begin{proof}[Proof of Lemma \ref{lem.semiconcTD}] Set $p^k =  D_{x^k}\mathcal V(t_0,\bx_0)$ and $p^t= \partial_t \mathcal V(t_0,\bx_0)$. Then, in view of  Lemma \ref{lem.semiconcesti}, we have,
for any $(t,\bx), (t_0,\bx_0)\in [0,T]\times (\R^d)^N$,
$$
\mathcal V^N(t,\bx) -\mathcal V^N(t_0, \bx_0) - \sum_{k=1}^N p^k\cdot (x^k-x^k_0) -p^t(t-t_0) \leq  \frac{C}{N} \sum_{k=1}^N |x^k-x^k_0|^2+ C(t-t_0)^2.
$$
The optimality of $(t_0,\bx_0,s_0,\by_0)$ also gives, for any $(t,\bx)$,
\be \label{firstorder}
\begin{split}
& \frac{1}{2\theta N} \sum_{i=1}^N |x^i-y^i_0|^2+\frac{1}{2\theta}(t-s_0)^2+{\mathcal V}^N(t,{\bf x})\\
& \geq
  \frac{1}{2\theta N} \sum_{i=1}^N |x^i_0-y^i_0|^2+\frac{1}{2\theta}(t_0-s_0)^2+{\mathcal V}^N(t_0,\bx_0).
\end{split}
\ee
From \eqref{firstorder}, we conclude that
$p^k=\dfrac{y^k_0-x^k_0}{\theta N} \ \ \text{and} \ \ p^t= \dfrac{s_0-t_0}{\theta}.$

Furthermore, rearranging \eqref{firstorder} yields
\begin{align*}
& \mathcal V^N(t,x) - \mathcal V^N(t_0,x_0) \geq \frac{1}{2 \theta N} \sum_{k=1}^N |x_0^k - y_0^k|^2 - \frac{1}{2 \theta N} \sum_{k=1}^N |x^k - y_0^k|^2 + \frac{1}{2 \theta N} |t_0 - s_0|^2\\
& - \frac{1}{2 \theta N} |t - s_0|^2
= \frac{1}{2 \theta N} \sum_{k=1}^N |x_0^k - y^k_0|^2 - \frac{1}{2 \theta N} - \sum_{k=1}^N| (x^k - x^k_0) + (x^k_0 - y^k_0)|^2\\ &
 + \frac{1}{2 \theta N} |t_0 - s_0|^2
- \frac{1}{2 \theta N} |(t - t_0) + (t_0 - s_0)|^2 \\
&= \sum_{k=1}^N p^k \cdot (x^k -x^k_0) + p^t (t - t_0) - \sum_{k=1}^N \frac{1}{2 \theta N} |x^k - x^k_0|^2 - \frac{1}{2 \theta} (t - t_0)^2.
\end{align*}
and, after some elementary manipulations,
$$
\mathcal V^N(t,\bx) -\mathcal V^N(t_0,\bx_0) - \sum_{k=1}^N p^k\cdot (x^k-x^k_0) -p^t(t-t_0) \geq  - \frac{1}{2\theta N} \sum_{k=1}^N |x^k-x^k_0|^2-\frac{1}{2\theta} (t-t_0)^2.
$$
Assuming that $\theta\leq (2C)^{-1}$, it follows that
\be\notag
\begin{split}
& w(t,\bx)= \mathcal V^N(t_0,\bx_0) - \mathcal V^N(t,\bx)\\ & + \sum_{k=1}^N p^k\cdot (x^k-x^k_0) +p^t(t-t_0)+ \frac{C}{N} \sum_{k=1}^N |x^k-x^k_0|^2+ C(t-t_0)^2
\end{split}
\ee
is convex and satisfies
$
0\leq w(t,\bx) \leq  \frac{1}{\theta N} \sum_{k=1}^N |x^k-x^k_0|^2+ \frac{1}{\theta}(t-t_0)^2.
$
Thus, for any $(t,\bx)$ and any $(s,\by)$, we have
\begin{align*}
& \sum_{k=1}^N D_{x^k}w(t,\bx)\cdot (y^k-x^k)+\partial_tw(t,\bx)(s-t) \leq w(t,\bx)\\
& +\sum_{k=1}^N D_{x^k}w(t,\bx)\cdot (y^k-x^k)+\partial_tw(t,\bx)(s-t) \\
&\qquad \leq w(s, \by) \leq   \frac{1}{\theta N} \sum_{k=1}^N |y^k-x^k_0|^2+\frac{1}{\theta}(s-t_0)^2.
\end{align*}
Letting  $y^k= x^k_0+\frac{1}{2}\theta ND_{x^k}w(t,\bx)$ and $s= t_0+ \frac{1}{2}\theta \partial_t w(t,x)$ in the inequality above, we  obtain
\be
\frac{\theta N}{4} \sum_{k = 1}^N \abs{ D_{x^k} w(t,\bx)}^2 \leq \sum_{k=1}^N D_{x^k}w(t,\bx) \cdot (x^k - x^k_0) + \partial_t w(t,x)(t - t_0),
\ee
and, after using the Cauchy-Schwarz inequality, 
\begin{align}
& \sum_{k=1}^N | D_{x^k}w(t,\bx)|\leq N^{1/2}\left(\sum_{k=1}^N | D_{x^k}w(t,\bx)|^2\right)^{1/2}\label{qlksjfdrfgbk}\\
&\qquad  \leq
 N^{1/2} \left( \frac{4}{N\theta} \sum_{k=1}^N |x^k_0-x^k| |D_{x^k}w(t,\bx)|+ \frac{4}{N\theta}|\partial_t w(t,x)||t-t_0|\right)^{1/2}.\notag
\end{align}
Recalling the definition of $w$ and that   $|D_{x^k}\mathcal V^N|\leq C/N$ and $|\partial_t \mathcal V^N|\leq C$, we find
$$
|D_{x^k}w(t,\bx)|= | - D_{x^k} \mathcal V^N(t,\bx)+ p^k+\frac{2C}{N} (x^k-x^k_0)|\leq CN^{-1}+\frac{2C}{N} |x^k-x^k_0|
$$
and
$$
|\partial_t w(t,x)|=| -\partial_t\mathcal V^N(t_0, \bx_0) +2 C(t-t_0)| \leq C.
$$
Returning  to \eqref{qlksjfdrfgbk},  we have
\be\notag
\begin{split}
& \sum_{k=1}^N | - D_{x^k} \mathcal V^N(t,\bx)+ p^k+\frac{2C}{N} (x^k-x^k_0)|\\ & \leq
  \left( \frac{C}{N\theta} \sum_{k=1}^N |x^k_0-x^k|+ \frac{C}{N\theta} \sum_{k=1}^N |x^k_0-x^k|^2+ \frac{C}{\theta}|t-t_0|\right)^{1/2},
\end{split}
\ee
from which we deduce the result by the definition of $p^k$.
 \end{proof}

\subsection{A concentration inequality} \label{subsec.ProofLemma}

This section is devoted to the proof of the following concentration inequality.

\begin{prop}
Take a constant drift $\alpha$ in $\R^d$, initial position $y_0^1,\dots,y_0^N$ in $\R^d$ for some $N \geq 1$ and consider $(Y_t^1)_{t\geq 0},\dots (Y_t^N)_{t\geq 0}$ defined, for all $1 \leq k \leq N$ and $t \geq 0$ by

$$Y^k_t = \alpha t + \sqrt{2}B^k_t,$$
where $(B_t^1)_{t\geq 0},\dots (B_t^n)_{t\geq 0}$ are $N$ independent d-dimensional standard Brownian motions defined on some probability space $(\Omega, \mathcal{F}, \mathbb{P})$. Define as well the empirical measure $\displaystyle m^N_{\bf{Y}_t} := \frac{1}{N}\sum_{k=1}^N \delta_{Y_t^k}$ and $m \in \mathcal{C}([0,T], \mathcal{P}_2(\R^d))$ to be the solution to 

\begin{equation}
\left \{
\begin{array}{ll}
\partial_t m + \alpha.Dm -\Delta m = 0 &\mbox{in } (0,+\infty) \times \R^d \\
m(0) = m^N_{\bf{Y}_0}.
\end{array}
\right.
\end{equation}
Then, there exists a positive constant $\beta \in (0,1/2)$ depending on the dimension $d$ and a positive constant $C$, depending on $|\alpha|$, $d$ and $T$ such that, for all $h \in [0,T]$ it holds

\begin{equation}
\E \left[ {\bf d}_1 (m(h), m^N_{{\bf Y}_h}) \right] \leq C( 1+ M_2^{1/2}(m^N_{\bf{Y}_0})) (h/N)^{\beta}.
\label{ConcentrationIneqGeneralFormulation}
\end{equation}

\label{ConcentrationInequalityGeneralCase}
\end{prop}

To prove Proposition \ref{ConcentrationInequalityGeneralCase}, it is convenient to introduce first a few facts and notations.

We denote by ${\bf L}$ denote the set of all 1-Lipschitz functions from $\R^d$ to $\R$, and let ${\bf L}_R$ be the set of all 1-Lipschitz functions $\phi : B_R \subset \R^d \to [-R,R]$.
For any $\phi \in {\bf L}_R$, we denote by $\tilde{\phi}$ the extension $\tilde{\phi} : \R^d \to [-R,R]$ given by
\begin{align*}
    \tilde{\phi}(x) = \begin{cases}
    \phi(x) & |x| \leq R, \\
    \frac{2R -|x|}{R} \phi(\frac{R}{|x|}x) & R < |x| < 2R, \\
    0 & |x| \geq 2R.
    \end{cases}
\end{align*}
Note that $\tilde{\phi}$ is also $1$-Lipschitz.

Let $\mathcal{L}(\epsilon, R)$ be the $\epsilon$-covering number of ${\bf L}_R$ with respect to the $\linf$-distance, that is,
\begin{align*}
    \mathcal{L}(\epsilon,R) = \inf \{k \in \N: \text{ there exist } \phi_1,...,\phi_k \in {\bf L}_R \text{ such that for all } \\
    \phi \in {\bf L}_R, \norm{\phi - \phi_j}_{\linf} < \epsilon \text{ for some } j\}.
\end{align*}
It is known (see, for example, \cite{kol61}) that
\begin{align} \label{covnumber}
    \mathcal{L}(\epsilon, 1) \leq \exp\{C\epsilon^{-d}\},
\end{align}
and, after a rescaling argument,
\begin{align} \label{rescaledcovnumber}
    \mathcal{L}(\epsilon, R) \leq \exp\{C\big(\frac{R}{\epsilon}\big)^d\}.
\end{align}
Indeed, if $\{\phi_1,...,\phi_n\}\in \bf L$ is $\epsilon/R$-dense in $\bf L$, then $\{\tilde{\phi}_1,...,\tilde{\phi}_n\}$ is $\epsilon$-dense in ${\bf L}_R$, where $\tilde{\phi}_i(x) = R \phi ({x}/{R})$. Thus \eqref{rescaledcovnumber} follows from \eqref{covnumber}.

To prove Proposition \ref{ConcentrationInequalityGeneralCase} we need two preliminary estimates.

we recall the notation after Lemma~\ref{lem.estixkyk}.

\begin{lem} \label{lem:tailest}
There exists a constant $C>0$ such that, for any $\phi \in {\bf L}$,
   $ \P[ \int_{\R^d} \phi (m(h) - m^N_{Y_h}) > x ] \leq \exp\left\{- \frac{Nx^2}{Ch}\right\}.$
\end{lem}

\begin{proof}
Let  $u$ be the solution of
$$
-\partial_t u -\Delta u -\alpha\cdot Du= 0\ \  {\rm in}\ \  (0,h)\times \R^d \ \ \text{and} \ \
u(h)=\phi\ \ {\rm in}\ \  \R^d,
$$
and note that, since $\|D\phi\|\leq 1$,   $\|Du\|_\infty\leq 1$.

Using It\^o's formula and the equation for $m$, we get
  $$  \int_{\R^d} \phi (m(h) - m^N_{Y_h}) =-\sqrt{2} \frac{1}{N} \sum_{k =1}^N \int_0^h Du(s,Y_s^k) .dB_s^k.$$
The random variables $h^{-1/2}\int_0^h Du(s,Y_s^k) dB_s^k$ are independent and sub-Gaussian, uniformly in $k$. Indeed, viewing $h^{-1/2}\int_0^{\cdot} D u(\cdot, Y^k) dB^k$ as a time-changed Brownian motion, we have that $B_\tau= h^{-1/2}\int_0^h D u(t, Y^k_t) dB^k_t $, where $B$ is a standard Brownian motion and $\tau \leq 1$ is a stopping time (we use here that $\|Du\|_\infty\leq 1$). In particular,
$
\P[\int_0^h Du(s,Y_s^k) dB_s^k > x] \leq \P[\sup_{0 \leq t \leq 1} |B_t| > h^{1/2}x],
$
from which the claim follows easily.

We  now apply Hoeffding's inequality (see, for example, Proposition 2.5 in \cite{Wainbook}) to complete the proof.
\end{proof}

\begin{lem} \label{lem:fixedjest}
There exists a constant $C$ such that, for any $R>0$, 
\begin{align*}
    \E[\sup_{ \phi \in {\bf L}_R} \int_{\R^d} \tilde{\phi} \big(m(h) - m^N_{Y_h}\big)] \leq C (1 + R^{\frac{d}{d+2}}) N^{\frac{-1}{d+2}}h^{\frac{1}{d+2}}.
 \end{align*}
\end{lem}

\begin{proof}
We fix $\epsilon > 0$ and use the estimate on $\mathcal{L}(\epsilon,R)$ to choose $K \leq \exp\{C\big(\frac{R}{\epsilon}\big)^d\}$ and $\phi_1,...,\phi_K$ in ${\bf L}_R$ such that, for each $\phi \in {\bf L}_R$, there exists $k \in \{1,...,K\}$ such that $\norm{\phi - \phi_k}_{\linf(B_R)} < \epsilon$, and hence $\norm{\tilde{\phi} - \tilde{\phi}_k}_{\linf(\R^d)} \leq \epsilon$.

Then, using Lemma \ref{lem:tailest} and the upper bound on $K$, for any $x>\epsilon$,  we have
\begin{align} \label{tailest}
    \P[\sup_{\phi \in {\bf L}_R} \int_{\R^d} \tilde{\phi} (m(h) - m^N_{Y_h}) > x] \leq \P[ \,\exists k \,\, \text{ such that } \int_{\R^d} \tilde\phi_k (m(h) - m^N_{Y_h}) > x - \epsilon] \nonumber \\
    \leq \sum_{k = 1}^K \P[ \int_{\R^d} \tilde\phi_k (m(h) - m^N_{Y_h}) > x - \epsilon]
    \leq \exp\left\{C \big(\frac{R}{\epsilon}\big)^d - \frac{N(x - \epsilon)^2}{Ch}\right\}.
 \end{align}

 We fix a small positive parameter $\gamma$, and note that, if 
   $  \epsilon = \gamma^{-\frac{1}{d}} R h^{1/d} x^{-2/d} N^{-1/d}$,
 then
 \begin{align} \label{comp}
     R\exp\left\{C \big(\frac{R}{\epsilon}\big)^d - \frac{Nx^2}{Ch}\right\} = R\exp\left\{C \gamma \frac{N x^2}{h} - \frac{N(x- \epsilon)^2}{Ch}\right\}.
 \end{align}
Further computations reveal that there is a constant $C$ such that $x> 2\epsilon$ as soon as
\be\label{takis20}
    x \geq C\dfrac{R^{\frac{d}{d+2}}h^{\frac{1}{d+2}}}{ \gamma^{\frac{1}{d+2}}{N^{\frac{1}{d+2}}}}. 
 \ee
By choosing $\gamma$ even smaller, we  deduce, in view of \eqref{tailest} and \eqref{comp},  that, for some constant $C$ and all $R$, $x$ as in \eqref{takis20},
$P[\sup_{\phi \in {\bf L}_R} \int_{\R^d} \tilde{\phi} (m(h) - m_{Y_h}) > x] \leq \exp\left\{ -\frac{N x^2}{Ch}\right\}.$
It follows that
\begin{align*}
    \E[\sup_{\phi \in {\bf L}_R} \int_{\R^d} \tilde{\phi} (m(h) - m^N_{Y_{h}})] & \leq \int_{0}^{C(\frac{R^d h}{{N}})^\frac{1}{d+2}}  1 dx + \int_{C(\frac{R^d h}{{N}})^\frac{1}{d+2}}^{\infty} \exp\{ \frac{-Nx^2}{Ch}\} dx\\[2mm]
&     \leq C (1 + R^{\frac{d}{d+2}}) N^{\frac{-1}{d+2}}h^{\frac{1}{d+2}}.
\end{align*}
\end{proof}

Finally, we give the proof of the concentration inequality.

 \begin{proof}[Proof of Proposition \ref{ConcentrationInequalityGeneralCase}]

 Throughout, $C$ is  a positive constant which, although changing from line to line, depends only on $d$, $T$, and $ |\alpha|$.

We fix $R > 0$, and note that, any $\psi \in {\bf L}$ normalized with $\psi(0) = 0$, can be written as $\psi = \tilde{\phi} + \varphi$, with  $\phi \in {\bf L}_R$ and $|\varphi| \leq |x| {\bf 1}_{B_R^c}$.

Thus,  for any $h\in(0,1]$,  we get
\be\label{maincomp}
\begin{split}
  & \E[{\bf d_1}(m(h),m^N_{\bf Y_h})] = \E[\sup_{\phi \in {\bf L}} \int_{\R^d} \phi (m(h) - m^N_{\bf Y_h})] \nonumber \\
  &  \leq \E[\sup_{\phi \in {\bf L}_R} \int_{\R^d} \tilde{\phi} (m(h) - m^N_{\bf Y_h})] + \int_{\R^d} |x|{\bf 1}_{B_R^c} m(h) + \E[\int_{\R^d} |x| {\bf 1}_{B_R^c} m^N_{\bf Y_h}] \nonumber \\
 &   \leq \E[\sup_{\phi \in {\bf L}_R} \int_{\R^d} \tilde{\phi} (m(h) - m^N_{\bf Y_h})] + \frac{M_2(m(h))}{R} + \frac{\E[M_2(m^N_{\bf Y_h})]}{R} \nonumber \\
 &   \leq \E[\sup_{\phi \in {\bf L}_R} \int_{\R^d}  \tilde{\phi} (m(h) - m^N_{\bf Y_h})] + C\frac{(1 +M_2(m(0)))}{R}.
\end{split}
\ee
Using Lemma \ref{lem:fixedjest}, we find that
\begin{align*}
    \E[{\bf d_1}(m(h),m^N_{Y_h})] \leq C (1 + R^{\frac{d}{d+2}}) N^{\frac{-1}{d+2}}h^{\frac{1}{d+2}}+ C\frac{(1 +M_2(m(0)))}{R} \\
    \leq C (1 + R) N^{\frac{-1}{d+2}}h^{\frac{1}{d+2}}+ C\frac{(1 +M_2(m(0)))}{R} .
\end{align*}
Optimizing in $R$, that is,  taking $R= N^{\frac{1}{2d+4}} h^{-\frac{1}{2d+4}} \sqrt{1 + M_2(m(0))}$, gives the result with $\beta =\frac{1}{2d+4}$.
\end{proof}

\section{The proof of Theorem \ref{thm.main1} with a common noise}\label{sec.CN}

 We now show that the method developed above can be adapted to problems with a common noise, that is, for  $a_0>0$. Recall that $\mathcal V^N$ and $\mathcal U$ are defined  by \eqref{def.VN} and \eqref{def.ucommon} respectively.

\begin{proof}[Proof of Theorem \ref{thm.main1} when $a_0>0$] Since the proof follows  closely  the one in the case $a^0=0$, here we emphasize and explain the main differences.

 We first note that the estimates of Lemma \ref{lem.estiVN} and \ref{lem.semiconcesti} remain valid (with the same proof), that is,  there exists $C>0$ such that
$
\|\mathcal V^N\|_\infty+N \sup_j \|D_{x^j}\mathcal V^N\|_\infty +\|\partial_t \mathcal V^N\|_\infty \leq C,
$
and, for any $(t,\bx)\in  [0,T]\times (\R^d)^N$,  $(\xi^i)_{i=1,\dots, N}\in (\R^d)^N$ and $\xi^0\in \R$,
\be\notag
\begin{split}
&\sum_{i,j=1}^ND^2_{x^ix^j}\mathcal V^N(t,{\bf x})\xi^i\cdot \xi^j  + 2 \sum_{i=1}^N D^2_{x^it}\mathcal V^N(t,\bx) \xi^i\xi^0+ D^2_{tt}\mathcal V^N (t,\bx)(\xi^0)^2\\
&\leq \frac{C}{N}\sum_{i=1}^N |\xi^i|^2+ C(\xi^0)^2.
\end{split}
\ee
We note for later use that the observation above   implies that the conclusion of Lemma \ref{lem.semiconcTD} still holds, because its proof  relies only on the above estimates.

However, the proof of Lemma \ref{It'^o} does not adapt to the case $a_0 > 0$. Hence, we need a new argument which relies on some results of \cite{DPT20}.


In particular, we have  the following analogue of Lemma \ref{It'^o}.

\begin{lem}\label{It'^ocommon} Assume \eqref{ass.main}.  There exists a constant $C>0$ depending only on the data such that, for all $s,t \in [0,T]$, all $r>2$ and all $m, m' \in \sP_r(\R^d)$
   $| \mathcal{U}(s,m) - \mathcal{U}(t,m') | \leq C \Big( \mathbf{d}_1(m,m') + |t - s|\Big),$
and, moreover, for any $\epsilon > 0$ and $(t_0,m_0)\in [0,T]\times \sP_2(\R^d)$, there exists  an  $\epsilon$-optimal control rule $\mathcal{R} = (\Omega, \sF, \mathbb{F}, \mathbb{P}, W, m,\alpha) \in \mathcal{A}(t_0,m_0)$ for $\mathcal{U}(t_0,m_0)$ such that
$
\|\alpha\|_\infty\leq C.
$
\end{lem}

\begin{proof}
Fix $R > 0$  and let $\mathcal{V}^{N,R}$ and $\mathcal{U}^R$ denote the values of the problems defining $\mathcal{V}^N$ and $\mathcal{U}$ when controls are restricted to the ball $B_R \subset \R^d$.

More precisely, define
$\mathcal{A}^{N,R}$ to be the set of $\alpha = (\alpha^k)_{k = 1}^N$'s  such that $|\alpha^k| \leq R$ for each $R$, and $\mathcal{A}^{R}(t_0,m_0)$ to be the set of $(\Omega, \sF, \mathbb{F}, \bP,W, m,\alpha) \in \mathcal{A}(t_0,m_0)$ such that $|\alpha| \leq R$. Then define $\mathcal{V}^{N,R}$ exactly as in \eqref{def.VN}  but with $\mathcal{A}^{N,R}$ replacing $\mathcal{A}$ and define $\mathcal{U}^R$ exactly as in \eqref{def.ucommon} but with $\mathcal{A}^R(t_0,m_0)$ replacing $\mathcal{A}^R$.

Then Proposition \ref{prop.equiv} and Theorem 3.6 of \cite{DPT20} give
   $ \lim_{N \to \infty} \mathcal{V}^{N,R}(t,\bx^N) = \mathcal{U}^R(t, m)$
where $\bx^N = (x^1,...,x^N)$, $m \in \sP_r(\R^d)$ and $x^1,...,x^N \in \R^d$ are such that
\begin{align*}
    \sup_{N} \frac{1}{N} \sum_{i = 1}^N |x^i|^r < \infty \ \ \text{and} \ \ \frac{1}{N} \sum_{i = 1}^N \delta_{x^i} \underset{N\to \infty} \to m \in \sP_2(\R^d).
\end{align*}


It follows from Lemma \ref{lem.estiVN} and Lemma \ref{lem.larger}, that there is $R_0>0$ such that $\mathcal{V}^{N,R_0} = \mathcal{V}^N$ and $\mathcal{U}^{R_0} = \mathcal{U}$, and so we infer that, for all $x_i$ and  $m$ as above,
    $\lim_{N \to \infty} \mathcal{V}^{N}(t,\bx^N) = \mathcal{U}(t, m).$
Hence, the uniform regularity on $\mathcal{V}^N$ established in \eqref{It'^o}, which,  as noted above,  holds equally well when $a_0 > 0$,  is enough to conclude that, for some $C>0$,
\newline
$|\mathcal{U}(s,m) - \mathcal{U}(t,m')| \leq C\big(d_1(m,m') + |t-s| \big) \quad \text{for all } m, m' \in \textcolor{blue}{\sP_r}(\R^d).$
Finally, for any $\epsilon > 0$ and  $(t_0,m_0)$, we can choose an $\epsilon$-optimal pair $(m,\alpha)$ for $\mathcal{U}^{R_0}$, and that this control is also $\epsilon$-optimal for $\mathcal{U}$. This completes the proof.
\end{proof}

Let $\hat{\mathcal V}^N$ be defined in Lemma \ref{lem.ineqeasy}. Then it is easily checked that $\hat{\mathcal V}^N$ is smooth and satisfies, with $\hat{\mathcal F}^N$ and $\hat{\mathcal G}^N$ as in Lemma \ref{lem.ineqeasy},
$$
\left\{\begin{array}{l}
\ds -\partial_t  \hat{\mathcal V}^N(t,m)-(1+a^0)\int_{\R^d}{\rm div}_y( D_m \hat{\mathcal V}^N(t,m,y))m(dy) \\[1.5mm]
\ds  \qquad - a_0\int_{\R^{2d}} {\rm tr}(D^2_{mm}\hat{\mathcal V}^N(t,x,m,y,y'))m(dy)m(dy')\\[1.5mm]
\ds \qquad +\int_{\R^d} H(y, D_m\hat{\mathcal V}^N (t,m,y))m(dy)\leq \hat{\mathcal F}(m)\ \  {\rm in} \ \  (0,T)\times \Pk,\\
\ds \hat{\mathcal V}^N(T,m)= \hat{\mathcal G}(m)\ \ {\rm in}\ \  \Pk,
\end{array}\right.
$$
Then, as in the proof of Lemma \ref{lem.ineqeasy}, it is possible to  use It\^o's formula for conditional measures (see, for example, \cite{CaDeBook} Book 2, Chapter 4) to infer that, for any solution $(m,\alpha)$ to \eqref{eq.mtcn}, 
$$
\hat{\mathcal V}^N(t_0,m_0) \leq \E\left[\int_{t_0}^T (\int_{\R^d} L(x, \alpha_t(x))m_t(x)dx + \hat{\mathcal F}^N(m_t))dt + \hat{\mathcal G}^N(m_T)\right].
$$
Using the same argument as in the proof of Lemma \ref{lem.ineqeasy} with Lemma \ref{It'^ocommon} replacing Lemma \ref{lem.estid1}, we arrive  at
$$
{\mathcal V}^N(t_0,m^N_{\bx_0})\leq \mathcal U(t_0,m^N_{\bx_0})+C(1+M_2^{1/2}(m^N_{\bx_0})) N^{-\beta}.
$$
\vs
We now turn to the opposite inequality. As before, for
 $\theta,\lambda\in (0,1)$, let
\be\notag
\begin{split}
M:&= \max_{(t,{\bf x}), (s,{\bf y})\in [0,T]\times (\R^d)^N} \big(e^{ s} (\mathcal U(s,m^N_{{\bf y}}) -{\mathcal V}^N(t,{\bf x}))\\
& -\frac{1}{2\theta N} \sum_i |x^i-y^i|^2  -\frac{1}{2\theta }|s-t|^2 -\frac{\lambda}{2N} \sum_{i=1}^N |y^i|^2,
\end{split}
\ee
and  denote by $((t_0,\bx_0),(s_0, \by_0))$ a maximum point in the expression above.

As in \eqref{lem.estixkyk} we have
\begin{equation} 
\frac{1}{N} \sum_{i=1}^N |x_0^{i} - y_0^{i}|^2 + |t_0-s_0|^2 \leq C \theta^2\qquad \frac{1}{N} \sum_{i=1}^N |y^i_0|^2\leq C\lambda^{-1}.
\label{estixkyk2}
\end{equation}
\vs

Next, for $\delta>0$, we use the partition $(C^j)_{j\in \{1, \dots, J\}}$ of $\{1, \dots, N\}$ constructed in Lemma~\ref{takis40}.
\vs

We set $\alpha^k=\bar \alpha^j$ if $k\in C^j$, and let
\begin{align*}
X^k_{s_0+\tau}= x^k_0+\tau \alpha^k &+ \sqrt{2} B^k_\tau+ \sqrt{2a^0}B^0_\tau, \ \  Y^k_{s_0+\tau}= y^k_0+\tau \alpha^k + \sqrt{2}B^k_{\tau}+ \sqrt{2a^0}B^0_\tau,\\
& \text{and} \ \ m^j_{{\bf Y}_{s_0+\tau}}= \dfrac{1}{n^j}\sum_{k\in C^j} \delta_{Y^k_{s_0+\tau}},
\end{align*}
and $m^j$ the solution to
$$
\left\{\begin{array}{l}
d m^j_t= \left[ (1+a^0) \Delta m^j_t -\bar\alpha^j \cdot Dm^j_t\right] + \sqrt{2a_0} Dm^j_t\cdot dB^0_t \ \ \ {\rm in}\ \  (s_0,T]\times \R^d\\[1.2mm] 
\ds m^j_{s_0}= m^j_{\by_0}  \ \ \text{in} \ \ \R^d.
\end{array}\right.
$$
Finally, we set $m_s= N^{-1} \sum_{j\in J} n^jm^j_s$, and claim that, for all  $h\geq 0$ and $ j\in \{1, \dots, J\}$,
\be \label{commonconc}
\E\left[  {\bf d}_1( m^j_{s_0+h}, m^j_{{\bf Y}_{s_0+h}}) \right] \leq  C(1+M_2^{1/2}(m^j_{s_0})){h^{\beta}}/{(n^j)^\beta},
\ee
and
\be \label{commonconc2}
\E\left[ {\bf d}_1( m_{s_0+h}, m^N_{{\bf X}_{t_0+h}}) \right] \leq  C\theta+  C\delta^{-d\beta}(1+ \lambda^{-\frac{1}{2}})h^{\beta}/{N^{\beta}}. 
\ee
The proof follows from Lemma \ref{lem.estid1} and estimate \eqref{estixkyk2}. Indeed, to establish \eqref{commonconc}, we first note that the process $(m_t)_{t\in [s_0,T]}$ solves \eqref{eq.mtcn} in the sense of distribution (with $B^0$ replacing $W$) if and only if the process $\tilde m_t=(Id-\sqrt{2a^0}(B_t^0-B_{t_0}))\sharp m_t$ solves $\P-$a.s. in the (classical) sense of distributions, with $\tilde \alpha_t(x)= \alpha_t(x+\sqrt{2a^0}(B_t^0-B_{t_0}^0)$,  the equation
\be\label{eq.mtcnBIS}
d\tilde m_t(x)= \left[\Delta \tilde m_t(x)-{\rm div}(\tilde m_t \tilde \alpha_t(x))\right]dt  \ \ \text{in} \ \ (t_0,T] \times \R^d \ \ \ \tilde m_{t_0}=m_0 \ \ \text{in} \ \ \R^d,
\ee
Next, we consider $$\tilde m^j_{t_0+\tau}= (Id  -  \sqrt{2a^0}B^0_\tau)\sharp m^j_{t_0+\tau} \ \ \text{and} \ \ \tilde{\bf{Y}}^k_{t_0+\tau}= {\bf Y}^k_{t_0+\tau}-\sqrt{2a^0}B^0_\tau,$$
and notice that $\tilde m^j$ and $\tilde{\bf{Y}}^k$ solve the same equations as in Lemma \ref{lem.estid1}, and, hence,  \eqref{commonconc} holds with $\tilde{m}^j_{t_0 + h}$ replacing $m^j_{t_0 + h}$ and $m^j_{\tilde{\bf{Y}}_{t_0 + h}}$ replacing $m^j_{\bf{Y}_{t_0 + h}}$.
\vs
Since $$m_{t_0 + r}^j = \tilde{m}_{t_0 + \tau}^j * \delta_{\sqrt{2 \alpha^0} B_{\tau}},$$ and
$$
    m^j_{\bf{Y}_{t_0 + \tau}}  = \dfrac{1}{n^j} \sum_{k \in C^j} \delta_{\tilde{\bf{Y}}_{t_0 + \tau}^k + \sqrt{2a^o} B^0_{\tau}} =  \Big(\dfrac{1}{n^j} \sum_{k \in C^j} \delta_{\tilde{\bf{Y}}_{s_0 + \tau}^k }\Big) * \delta_{\sqrt{2a^o} B^0_{\tau}} = \tilde{m}^j_{\bf{Y}_{s_0 + \tau}} * \delta_{\sqrt{2a^o} B^0_{\tau}},
$$
we can conclude that
\be\notag
\begin{split}
\E\left[  {\bf d}_1( m^j_{s_0+h}, m^j_{{\bf Y}_{s_0+h}}) \right]& = \E\left[  {\bf d}_1( m^j_{s_0+h} * \delta_{\sqrt{2\alpha^0} B^0_h}, \,  m^j_{{\bf Y}_{s_0+h}} * \delta_{\sqrt{2\alpha^0} B^0_h}) \right] \\
&= \E\left[  {\bf d}_1( \tilde{m}^j_{s_0+h}, m^j_{{\tilde{\bf{Y}}}_{s_0+h}}) \right],
\end{split}
\ee
and so \eqref{commonconc} holds. The proof for \eqref{commonconc2} is similar.

We proceed with the proof noticing that the dynamic programming principle in Lemma \ref{lem.PDD} still holds but with an expectation, since now the measures are random, and with Proposition \ref{prop.dpcommon} replacing Proposition \ref{pro.dppnocommon}.

Moreover, since the conclusion of Lemma \ref{lem.semiconcTD} also holds as already pointed out, we can argue as in the proof of Proposition \ref{prop.hardineq} (the time-regularity provided by Lemma \ref{It'^ocommon} replacing that in Lemma \ref{It'^o}) that
$
\mathcal U(t,m^N_{{\bf x}}) -{\mathcal V}^N(t,{\bf x})) \leq C N^{-\beta}(1+ \frac{1}{N} \sum_{i=1}^N |x^i|^2).
$
The conclusion then follows as in the proof of Theorem \ref{thm.main1}.
\end{proof}

\section{Appendix}
We adapt some technical results from \cite{DPT19} and \cite{DPT20} to our setting. Most importantly, we infer the dynamic programming principle (Proposition \ref{prop.dpcommon}) in our setting from the dynamic programming principle which is stated in Theorem 3.1 of \cite{DPT20}. Most of the arguments are straightforward adaptations of the superposition and mimicking results achieved in \cite{LSZ20}, and so the proofs are only sketched.
%

Following Definition 2.1 in \cite{DPT19} and Definition 2.3 \cite{DPT20}  we define, for each $(t_0,m_0) \in [0,T] \times \sP_2(\R^d)$, the set of weak controls $\mathcal{A}_w(t_0,m_0)$ to be the set of tuples
   $ \mathcal{R} = (\Omega, \sF, \bP, \mathbb{F} = (\sF_t)_{0 \leq t \leq T}, \mathbb{G} = (\sG)_{0 \leq t \leq T}, X, B, W, m, \alpha)$
such that
\begin{enumerate}
    \item $(\Omega, \sF, \bP)$ is a probability space equipped with filtrations $\mathbb{G}$, $\mathbb{F}$ such that, for all  $0 \leq t \leq T$, $\sG_t \subset \sF_t$ and  $\sF_t \vee \sF_T^B \perp \sG_T | \sG_t.$
    \item $X = (X_t)_{0 \leq t \leq T}$ is a continuous, $\mathbb{F}$-adapted $\R^d$ valued process.
    \item $\alpha = (\alpha_t)_{t_0 \leq t \leq T}$ is a bounded, $\mathbb{F}$-predictable process taking values in $\R^d$.
    \item $(B,W)$ is a $\R^d \times \R^d$-valued standard $\mathbb{F}$ Brownian motion, $W$ is $\mathbb{G}$-adapted, and
    $\sF_t \vee \sigma(B) \perp \sG_T$.
    \item $m = (m_t)_{t_0 \leq t \leq T}$ is a $\mathbb{G}$-predictable process taking values in $\sP_2(\R^d)$ and such that $m_t = \sL(X_t | \sG_t)$ for $d\bP \otimes ds$-a.e. $(s,\omega) \in [t,T] \times \Omega$.
    \item For all $t_0 \leq t \leq T$
    \begin{align*}
        X_t = X_{t_0} + \int_{t_0}^t \alpha_s ds + \sqrt{2} (B_t - B_{t_0}) + \sqrt{2a_0} (W_t - W_{t_0}), \quad \sL(X_{t_0}) = m_0. 
    \end{align*}
\end{enumerate}
We also let
\begin{align*}
    \mathcal{U}_w(t_0,m_0) := \inf_{\mathcal{R} \in \mathcal{A}_w(t_0,m_0)} \E^{\bP}[\int_{t_0}^T (L(X_t, \alpha_t) + \sF(m_t) ) dt + \sG(m_T)]
\end{align*}

In our context, a superposition principle is a result asserting the following: given a control rule $\mathcal{R} = (\Omega, \sF, \mathbb{F}, \bP, W, m, \alpha) \in \mathcal{A} \in \mathcal{R}(t_0,m_0)$, we can find an extension $(\tilde{\Omega}, \tilde{\sF}, \mathbb{G})$ of $(\Omega, \sF, \mathbb{F})$ hosting another Brownian motion $B$ independent of $\mathbb{F}$  and a process $X$ such that $dX_t= \alpha_t(X_t) dt +\sqrt{2}dB_t + \sqrt{2a_0}dW_t$ such that $m_t = \sL(X_t | \sF_t)$. We refer to \cite{LSZ20} for details. The superposition results of \cite{LSZ20} are useful to us because we need to apply some technical results from \cite{DPT19, DPT20}, and the superposition allows us to check that our formulation is equivalent to the one used in \cite{DPT19,DPT20}.
\vs
In what follows, for technical reasons, that is, to have  the coercivity condition on the cost appearing in Assumption 2.1 of \cite{DPT20}, we will work with a truncated version of the weak formulation defined here. Namely, we define $\mathcal{A}^R_w(t_0,m_0)$ just as $\mathcal{A}_w(t_0,m_0)$, but with the controls $\alpha$ required to take values in $B_R \subset \R^d$. Then, we write
\begin{align*}
    \mathcal{U}^R_w(t_0,m_0) := \inf_{\mathcal{R} \in \mathcal{A}^R_w(t_0,m_0)} \E^{\bP}[\int_{t_0}^T (L(X_t, \alpha_t) + \sF(m_t) ) dt + \sG(m_T)]
\end{align*}
We also truncate the original form of the problem, by defining $\mathcal{U}^R$ just like $\mathcal{U}$, but with controls $\alpha$ required to take values in $B_R \subset \R^d$.

The following can be obtained using the superposition and following results of \cite{LSZ20}, as in the proof of Theorem 8.3 of \cite{LSZ20}.

\begin{prop} \label{prop.equiv}
For each $R$,  $\mathcal{U}^R_w = \mathcal{U}^R$.
\end{prop}
It is also useful to note that the  regularity results of  Lemma \ref{lem.estiVN}, which holds also in the case $a_0 > 0$,  can be used to infer that $\mathcal{U}^{R} = \mathcal{U}$ for all $R \geq R_0$.

\begin{lem} \label{lem.larger}
There exists $R_0$ depending on the data such that, for each $R \geq R_0$, $\mathcal{U}^R = \mathcal{U}$.
\end{lem}
\begin{proof}
Theorem 3.1 and Theorem 3.6 in \cite{DPT20} together with Proposition \ref{prop.equiv} yield that, 


for all  $t \in [0,T]$, all $r>2$, all $m \in \sP_r(\R^d)$ and $x_i \in \R^d$ such that, for some $\epsilon >0$,
$$\sup_{N} \dfrac{1}{N} \sum_{i =1}^N |x_i|^{r} < \infty \ \ \text{and} \ \dfrac{1}{N} \sum_{i = 1}^N \delta_{x_i} \underset{N\to \infty} \to m \ \ \text{ in} \ \ \ \sP_2(\R^d),$$

we have, for  $\bx^N = (x_1,...,x_N) \in (\R^d)^N$,
\begin{align} \label{pdtconv}
    \lim_{N \to \infty} \mathcal{V}^{R,N}(t,\bx^N) = \mathcal{U}^R(t, m).
\end{align}

Next, notice that, by \eqref{lem.estiVN} (see Remark \ref{rmk.bddcontrol}), there is  $R_0$ depending only on the data such that, for all $R \geq R_0$, $\mathcal{V}^{N,R} = \mathcal{V}^{N}$. Thus \eqref{pdtconv} actually gives, for all $R \geq R_0$,
   $ \lim_{N \to \infty} \mathcal{V}^{N}(t,\bx^N) = \mathcal{U}^R(t, m).$
It follows that $\mathcal{U} = \mathcal{U}^{R_0}.$
Indeed, clearly $\mathcal{U} \leq \mathcal{U}^{R_0}$.

For the other inequality, for any $(t_0,m_0)$, we can choose $\mathcal{R} = (\Omega, \sF, \mathbb{F}, \bP, W, m, \alpha)$ to be $\epsilon$-optimal in the definition of $\mathcal{U}(t_0,m_0)$. Since $\alpha$ is bounded by hypothesis, there exist $R \geq R_0$ such that $\mathcal{R} \in \mathcal{A}^{R}(t_0,m_0)$, and, hence,
   $ \mathcal{U}^{R_0}(t_0,m_0) = \mathcal{U}^R(t_0,m_0) \leq \mathcal{U}(t_0,m_0) + \epsilon.$
Letting  $\epsilon \to 0$ gives $\mathcal{U}(t_0,m_0) = \mathcal{U}^{R_0}(t_0,m_0)$.
\end{proof}
Now, we turn to the dynamic programming principle, that is, Proposition \ref{prop.dpcommon}.
\begin{proof}[Proof of Proposition \ref{prop.dpcommon}]
We combine Theorem 3.1 of \cite{DPT19} with Proposition \ref{prop.equiv} to conclude that, for all $0 \leq t_0 \leq t_1 \leq T$ and any $R \geq R_0$,
\begin{align*}
    \mathcal{U}(t_0,m_0) &= \mathcal{U}^{R}(t_0,m_0) \\
    &= \mathcal{U}^{R}_w(t_0,m_0) = \inf_{\mathcal{R} \in \mathcal{A}^{R}_w(t_0,m_0)} \E^{\bP}[ \int_{t_0}^{t_1} (L(X_t, \alpha_t) + \sF(m_t)) dt + \mathcal{U}^{R}_W(t_1, m_{t_1})] \\
   & =  \inf_{\mathcal{R} \in \mathcal{A}^{R}_w(t_0,m_0)} \E^{\bP}[ \int_{t_0}^{t_1} (L(X_t, \alpha_t) + \sF(m_t)) dt + \mathcal{U}^{R}(t_1, m_{t_1})] \\
 &   = \inf_{\mathcal{R} \in \mathcal{A}^{R}_w(t_0,m_0)} \E^{\bP}[ \int_{t_0}^{t_1} (L(X_t, \alpha_t) + \sF(m_t)) dt + \mathcal{U}(t_1, m_{t_1})].
\end{align*}
Since $R$ can be arbitrarily large, it is easy to see that the above imply
\[    \mathcal{U}(t_0,m_0) = \inf_{\mathcal{R} \in \mathcal{A}_w(t_0,m_0)} \E^{\bP}[ \int_{t_0}^{t_1} (L(X_t, \alpha_t) + \sF(m_t)) dt + \mathcal{U}(t_1, m_{t_1})],\]
and,  using the superposition and adapting arguments from \cite{LSZ20}, Proposition~\ref{prop.dpcommon} follows. 
\end{proof}

\bibliographystyle{siam}

\end{document}